\documentclass[letterpaper,reqno,12pt]{amsart}
\usepackage[margin=1.2in]{geometry}

\usepackage{amssymb, enumerate,stmaryrd,mathrsfs}
\usepackage[all]{xy}
\usepackage{hyperref}
\hypersetup{colorlinks=true}

\theoremstyle{plain}
\newtheorem{thm}{Theorem}[section] 
\newtheorem{cor}[thm]{Corollary}
\newtheorem{prop}[thm]{Proposition}

\newtheorem{lem}[thm]{Lemma}

\newtheorem{question}{Question}

\theoremstyle{definition} 
\newtheorem{defn}[thm]{Definition}

\newtheorem{setting}[thm]{Setting}
\newtheorem{eg}[thm]{Example}

\newtheorem{defn-prop}[thm]{Definition-Proposition}

\theoremstyle{remark}
\newtheorem{rem}[thm]{Remark}

\newtheorem*{acknowledgement}{Acknowledgments}

\def\ge{\geqslant}

\def\phi{\varphi}
\def\epsilon{\varepsilon}
\def\tilde{\widetilde}
\def\bar{\overline}
\def\mapsto{\longmapsto}

\newcommand{\sO}{\mathcal{O}}
\newcommand{\J}{\mathcal{J}}
\newcommand{\sL}{\mathcal{L}}
\newcommand{\B}{\mathcal{B}}

\newcommand{\F}{\mathbb{F}}
\newcommand{\N}{\mathbb{N}}
\newcommand{\Q}{\mathbb{Q}} 
\newcommand{\C}{\mathbb{C}} 
\newcommand{\R}{\mathbb{R}} 
\newcommand{\Z}{\mathbb{Z}}

\newcommand{\x}{\mathbf{x}}

\newcommand{\ba}{\mathfrak{a}}
\newcommand{\m}{\mathfrak{m}}

\newcommand{\p}{\mathfrak{p}}

\newsavebox{\circlebox}
\savebox{\circlebox}{\fontencoding{OMS}\selectfont\Large\char13}
\newlength{\circleboxwdht}

\def\Spec{\operatorname{Spec}}

\def\Div{\operatorname{div}}

\def\Ker{\operatorname{Ker}}
\def\Ann{\operatorname{Ann}}

\def\Exc{\operatorname{Exc}}

\def\id{\operatorname{id}}
\def\cl{\operatorname{cl}}

\def\gen{\operatorname{gen}}

\DeclareMathOperator*{\ulim}{ulim}
\title[BCM test ideals in equal characteristic zero via ultraproducts]{Big Cohen-Macaulay test ideals in equal characteristic zero via ultraproducts}

\author{Tatsuki Yamaguchi}
\address{Graduate School of Mathematical Sciences, University of Tokyo, 3-8-1 Komaba, Meguro-ku, Tokyo 153-8914, Japan}
\email{tyama@ms.u-tokyo.ac.jp}

\subjclass[2020]{14F18, 14B05}
\keywords{Multiplier ideals, Big Cohen-Macaulay algebras, Ultraproducts}

\begin{document}

\begin{abstract}
	 Utilizing ultraproducts, Schoutens constructed a big Cohen-Macaulay algebra $\B(R)$ over a local domain $R$ essentially of finite type over $\C$. We show that if $R$ is normal and $\Delta$ is an effective $\Q$-Weil divisor on $\Spec R$ such that $K_R+\Delta$ is $\Q$-Cartier, then the BCM test ideal $\tau_{\widehat{\B(R)}}(\widehat{R},\widehat{\Delta})$ of $(\widehat{R},\widehat{\Delta})$ with respect to $\widehat{\B(R)}$ coincides with the multiplier ideal $\J(\widehat{R},\widehat{\Delta})$ of $(\widehat{R},\widehat{\Delta})$, where $\widehat{R}$ and $\widehat{\B(R)}$ are the $\m$-adic completions of $R$ and $\B(R)$, respectively, and $\widehat{\Delta}$ is the flat pullback of $\Delta$ by the canonical morphism $\Spec \widehat{R}\to\Spec R$. As an application, we obtain a result on the behavior of multiplier ideals under pure ring extensions.
\end{abstract}

\maketitle


\section{Introduction}
  A (balanced) big Cohen-Macaulay algebra over a Noetherian local ring $(R,\m)$ is an $R$-algebra $B$ such that every system of parameters is a regular sequence on $B$. Its existence implies many fundamental homological conjectures including the direct summand conjecture (now a theorem). Hochster and Huneke \cite{HH92}, \cite{HH95} proved the existence of a big Cohen-Macaulay algebra in equal characteristic, and Yves Andr\'{e} \cite{And} settled the mixed characteristic case. Recently, using big Cohen-Macaulay algebras, Ma and Schwede \cite{Ma-Schwede regular}, \cite{Ma-Schwede} introduced the notion of BCM test ideals as an analog of test ideals in tight closure theory.
  
  The test ideal $\tau(R)$ of a Noetherian local ring $R$ of positive characteristic was originally defined as the annihilator ideal of all tight closure relations of $R$.
  Since it turned out that $\tau(R)$ was related to multiplier ideals via reduction to characteristic $p$, the definition of $\tau(R)$ was generalized in \cite{HY}, \cite{interpret. multiplier} to involve effective $\Q$-Weil divisors $\Delta$ on $\Spec R$ and ideals $\ba\subseteq R$ with real exponent $t>0$. In these papers, it was shown that multiplier ideals coincide, after reduction to characteristic $p \gg 0$, with such generalized test ideals $\tau(R,\Delta,\ba^t)$. In positive characteristic, Ma-Schwede's BCM test ideals are the same as the generalized test ideals. In this paper, we study BCM test ideals in equal characteristic zero.
  
  Using ultraproducts, Schoutens \cite{log-terminal} gave a characterization of log-terminal singularities, an important class of singularities in the minimal model program. He also gave an explicit construction of a big Cohen-Macaulay algebra $\B(R)$ in equal characteristic zero: $\B(R)$ is described as the ultraproduct of the absolute integral closures of Noetherian local domains of positive characteristic. He defined a closure operation associated to $\B(R)$ to introduce the notions of $\B$-rationality and $\B$-regularity, which are closely related to BCM rationality and BCM regularity defined in \cite{Ma-Schwede}, and proved that $\B$-rationality is equivalent to being rational singularities. The aim of this paper is to give a geometric characterization of BCM test ideals associated to $\B(R)$. Our main result is stated as follows:
   \begin{thm}[Theorem \ref{Main Theorem}]\label{introduction main theorem}
  	Let $R$ be a normal local domain essentially of finite type over $\C$. Let $\Delta$ be an effective $\Q$-Weil divisor on $\Spec R$ such that $K_R+\Delta$ is $\Q$-Cartier, where $K_R$ is a canonical divisor on $\Spec R$. Suppose that $\widehat{R}$ and $\widehat{\B(R)}$ are the $\m$-adic completions of $R$ and $\B(R)$, and $\widehat{\Delta}$ is the flat pullback of $\Delta$ by the canonical morphism $\Spec \widehat{R}\to\Spec R$. Then we have
  	\[
  		\tau_{\widehat{\B(R)}}(\widehat{R},\widehat{\Delta})=\J(\widehat{R},\widehat{\Delta}),
  	\]
  	where $\tau_{\widehat{\B(R)}}(\widehat{R},\widehat{\Delta})$ is the BCM test ideal of $(\widehat{R},\widehat{\Delta})$ with respect to $\widehat{\B(R)}$ and $\J(\widehat{R},\widehat{\Delta})$ is the multiplier ideal of $(\widehat{R},\widehat{\Delta})$.
  \end{thm}
  The inclusion $\J(\widehat{R},\widehat{\Delta}) \subseteq \tau_{\widehat{\B(R)}}(\widehat{R},\widehat{\Delta})$ is obtained by comparing reductions of the multiplier ideal modulo $p\gg 0$ to its approximations. We prove the opposite inclusion by combining an argument similar to that in \cite{Puresubrings} with the description of multiplier ideals as the kernel of a map between local cohomology modules in \cite{interpret. multiplier}.
 As an application of Theorem \ref{introduction main theorem}, we show the next result about a behavior of multiplier ideals under pure ring extensions, which is a generalization of \cite[Corollary 5.30]{TY}.
\begin{thm}[Corollary \ref{Main Thm 2}] \label{introduction main theorem 2}
 	Let $R\hookrightarrow S$ be a pure local homomorphism of normal local domains essentially of finite type over $\C$. Suppose that $R$ is $\Q$-Gorenstein. Let $\Delta_S$ be an effective $\Q$-Weil divisor such that $K_S+\Delta_S$ is $\Q$-Cartier, where $K_S$ is a canonical divisor on $\Spec S$. Let $\ba\subseteq R$ be a nonzero ideal and $t>0$ a positive rational number. Then we have
 	\[
 	\J(S,\Delta_S,(\ba S)^t)\cap R\subseteq \J(R,\ba^t).
 	\]
 \end{thm}
 In \cite{TY}, we defined ultra-test ideals, a variant of test ideals in equal characteristic zero, to generalize the notion of ultra-$F$-regularity introduced by Schoutens \cite{log-terminal}. Theorem \ref{introduction main theorem 2} was proved by using ultra-test ideals under the assumption that $\ba$ is a principal ideal. The description of multiplier ideals as BCM test ideals associated to $\B(R)$ (Theorem \ref{introduction main theorem}) and a generalization of module closures in \cite{PRG21} enables us to show Theorem \ref{introduction main theorem 2} without any assumptions.
 
 As another application of Theorem \ref{introduction main theorem}, we give an affirmative answer to one of the conjectures proposed by Schoutens \cite[Remark 3.10]{log-terminal}, which says that $\B$-regularity is equivalent to being log-terminal singularities (see Theorem \ref{partial answer to Schoutens}).
 
 This paper is organized as follows:
 in the preliminary section, we give definitions of multiplier ideals, test ideals and BCM test ideals. In Section 3, we quickly review the theory of ultraproducts in commutative algebra including non-standard and relative hulls. In Section 4, we prove some fundamental results on big Cohen-Macaulay algebras constructed via ultraproducts following \cite{canonical BCM}. In Section 5, we review the relation between approximations and reductions modulo $p\gg0$ and consider approximations of multiplier ideals. In Section 6, we show Theorem \ref{introduction main theorem}, the main theorem of this paper.
  In Section 7, using a generalized module closure, we show Theorem \ref{introduction main theorem 2} as an application of Theorem \ref{introduction main theorem}. In Section 8, we show that $\B$-regularity is equivalent to log-terminal singularities.
 Finally in Section 9, we discuss a question, a variant of \cite[Question 2.7]{seed}, to handle big Cohen-Macaulay algebras that cannot be constructed via ultraproducts, and consider the equivalence of BCM-rationality and being rational singularities.
\begin{small}
\begin{acknowledgement}
The author wishes to express his gratitude to his supervisor Professor Shunsuke Takagi for his valuable advice and suggestions. The author would also like to thank Hans Schoutens for helpful comments on this paper. He also thanks to the referee who  provided useful comments and suggestions. The author was supported by JSPS KAKENHI Grant Number JP22J13150.
\end{acknowledgement}
\end{small}
\section{Preliminaries}
 Through out this paper, all rings will be commutative with unity.
\subsection{Multiplier ideals}
Here, we briefly review the definition of multiplier ideals and refer the reader to \cite{La},\cite{ST21} for more details.
Throughout this subsection, we assume that $X$ is a normal integral scheme essentially of finite type over a field of characteristic zero or $X=\Spec \widehat{R}$, where $(R,\m)$ is a normal local domain essentially of finite type over a field of characteristic zero and $\widehat{R}$ is its $\m$-adic completion.
\begin{defn}
	A proper birational morphism $f:Y\to X$ between  integral schemes is said to be a {\it resolution of singularities} of $X$ if $Y$ is regular. When $\Delta$ is a $\Q$-Weil divisor on $X$ and $\ba \subseteq \sO_X$ is a nonzero coherent ideal sheaf, a resolution $f:Y \to X$ is said to be a {\it log resolution} of $(X,\Delta,\ba)$ if $\ba\sO_Y=\sO_Y(-F)$ is invertible and if the union of the exceptional locus $\Exc (f)$ of $f$ and the support $F$ and the strict transform $f_*^{-1}\Delta$ of $\Delta$ is a simple normal crossing divisor.
\end{defn}
If $f:Y\to X$ is a proper birational morphism with $Y$ a normal integral scheme and $\Delta$ is a $\Q$-Weil divisor, then we can choose $K_Y$ such that $f^*(K_X+\Delta)-K_Y$ is a divisor supported on the exceptional locus of $f$. With this convention:
\begin{defn}
	Let $\Delta\ge 0$ be an effective $\Q$-Weil divisor on $X$ such that $K_X+\Delta$ is $\Q$-Cartier, $\ba\subseteq \sO_X$ a nonzero coherent ideal sheaf and $t>0$ a positive real number.
	Then the {\it multiplier ideal sheaf} $\J(X,\Delta,\ba^t)$ associated to $(X,\Delta,\ba^t)$ is defined by
	\[
		\J (X,\Delta,\ba^t)=f_*\sO_Y(K_Y-\lfloor f^*(K_X+\Delta)+tF\rfloor).
	\]
	where $f:Y\to X$ is a log resolution of $(X,\Delta,\ba)$. Note that this definition is independent of the choice of log resolution.
\end{defn}
 \begin{defn}
 	Let $X$ be a normal integral scheme essentially of finite type over a field of characteristic zero. We say that {\it $X$ has rational singularities} if $X$ is Cohen-Macaulay at $x$ and if for any projective birational morphism $f:Y\to \Spec \sO_{X,x}$ with $Y$ a normal integral scheme, the natural morphism $f_*\omega_Y\to\omega_{X,x}$ is an isomorphism. 
 \end{defn}
\subsection{Tight closure and Test ideals}
In this subsection, we quickly review the basic notion of tight closure and test ideals. We refer the reader to \cite{BSTZ}, \cite{HH}, \cite{HY}, \cite{interpret. multiplier}.
\begin{defn}
	Let $R$ be a normal domain of characteristic $p>0$, $\Delta\ge 0$ an effective $\Q$-Weil divisor, $\ba\subseteq R$ a nonzero ideal and $t>0$ a real number. Let $E=\bigoplus E(R/\m)$ be the direct sum, taken over all maximal ideals $\m$ of $R$, of the injective hulls $E_R(R/\m)$ of the residue  fields $R/\m$. 
	\begin{enumerate}
		\item Let $I$ be an ideal of $R$. The {\it $(\Delta,\ba^t)$-tight closure $I^{*\Delta,\ba^t}$ of $I$} is defined as follows:
		$x\in I^{*\Delta,\ba^t}$ if and only if there exists a nonzero element $c\in R^{\circ}$ such that
		\[
		c\ba^{\lceil t(q-1)\rceil}x^q\subseteq I^{[q]}R(\lceil (q-1)\Delta\rceil)
		\]
		for all large $q=p^e$, where $I^{[q]}=\{f^q|f\in I\}$ and $R^{\circ}=R\setminus\{0\}$.
		\item If $M$ is an $R$-module, then the {\it $(\Delta, \ba^t)$-tight closure $0_M^{*\Delta,\ba^t}$} is defined as follows:
		$z\in 0_M^{*\Delta, \ba^t}$ if and only if there exists a nonzero element $c\in R^{\circ}$ such that
		\begin{equation*}
			(c\ba^{\lceil t(q-1)\rceil})^{1/q}\otimes z=0 \quad \text{in} \quad R(\lceil (q-1)\Delta\rceil)^{1/q}\otimes_R M
		\end{equation*}
		for all large $q=p^e$.
		\item The {\it (big) test ideal $\tau (R,\Delta,\ba^t)$ associated to $(R,\Delta,\ba^t)$} is defined by
		\[
		\tau (R,\Delta,\ba^t)=\Ann_R(0_E^{*\Delta,\ba^t}).
		\]
		When $\ba=R$, then we simply denote the ideal $\tau(R,\Delta)$. We call the triple $(R,\Delta,\ba^t)$ is {\it strongly $F$-regular} if $\tau(R,\Delta,\ba^t)=R$. 
	\end{enumerate}
\end{defn}
\begin{defn}[{\cite{FW}}]
	Let $R$ be an $F$-finite Noetherian local domain of characteristic $p>0$ of dimension $d$. We say that $R$ is {\it $F$-rational} if any ideal $I=(x_1,\dots,x_d)$ generated by a system of parameters satisfies $I=I^*$.
\end{defn}


\subsection{Big Cohen-Macaulay algebras}
 In this subsection, we will briefly review the theory of big Cohen-Macaulay algebras.
 Throughout this subsection, we assume that local rings $(R,\m)$ are Noetherian. 
 \begin{defn}
 	Let $(R,\m)$ be a local ring, and let $\x = x_1,\dots, x_n$ be a system of parameters. $R$-algebra $B$ is said to be {\it big Cohen-Macaulay with respect to $\x$} if $\x$ is a regular sequence on $B$. $B$ is called a {\it (balanced) big Cohen-Macaulay algebra} if it is big Cohen-Macaulay with respect to $\x$ for every system of parameters $\x$.
\end{defn}
\begin{rem}[{\cite[Corollary 8.5.3]{BH93}}]
	If $B$ is big Cohen-Macaulay with respect to $\x$, then the $\m$-adic completion $\widehat{B}$ is (balanced) big Cohen-Macaulay.
\end{rem}
 About the existence of big Cohen-Macaulay algebras of residue characteristic $p>0$, the following are proved in \cite{Bha20}, \cite{HH92}.
\begin{thm}\label{absolute integral closure p>0}
	If $(R,\m)$ is an excellent local domain of residue characteristic $p>0$, then the $p$-adic completion of absolute integral closure $R^+$ is a (balanced) big Cohen-Macaulay $R$-algebra.
\end{thm}
Using big Cohen-Macaulay algebras, we can define a class of singularities.
\begin{defn}
	If $R$ is an excellent local ring of dimension $d$ and let $B$ be a big Cohen-Macaulay $R$-algebra. We say that $R$ is {\it big Cohen-Macaulay-rational with respect to $B$} (or simply $\text{BCM}_B$-rational) if $R$ is Cohen-Macaulay and if $H_\m^d(R)\to H_\m^d(B)$ is injective. We say that $R$ is {\it BCM-rational} if $R$ is $\text{BCM}_B$-rational for any big Cohen-Macaulay algebra $B$.
\end{defn}
 We explain BCM test ideals introduced in \cite{Ma-Schwede}.
 \begin{setting} \label{Setting} Let $(R,\m)$ be a normal local domain of dimension d.
 	\begin{enumerate}[(i)]
 		\item $\Delta \ge 0$ is a $\Q$-Weil divisor on $\Spec R$ such that $K_R+\Delta$ is $\Q$-Cartier.
 		\item Fixing $\Delta$, we also fix an embedding $R\subseteq \omega_R \subseteq \operatorname{Frac} R$, where $\omega_R$ is the canonical module.
 		\item Since $K_R + \Delta$ is effective and $\Q$-Cartier, there exist an integer $n>0$ and $f\in R$ such that $n(K_R+\Delta)=\Div (f)$.
 	\end{enumerate}
 \end{setting}
 \begin{defn}\label{def of BCM test ideal}
 	With notation as in Setting \ref{Setting}, if $B$ is a big Cohen-Macaulay $R[f^{1/n}]$-algebra, then we define $0_{H_\m^d(\omega_R)}^{B,K_R+\Delta}$ to be $\Ker \psi $, where $\psi$ is the homomorphism determined by the below commutative diagram:
 	\[
 	\xymatrix{
 		H_\m^d(R) \ar[r] \ar[d] & H_\m^d(B) \ar[r]^{\cdot f^{1/n}} \ar[d] & H_\m^d(B) \\
 		H_\m^d(\omega_R) \ar[r] \ar@/_45pt/[rru]_{\psi} & H_\m^d(B\otimes_R \omega_R) \ar[ru]
 	}.
 	\]
 	If $R$ is $\m$-adically complete, then we define 
 	\[\tau_B(R, \Delta)=\Ann_{R} 0_{H_\m^d(\omega_R)}^{B,K_R+\Delta}.
 	\]
 	We call {\it $\tau_B(R,\Delta)$ the BCM test ideal of $(R,\Delta)$ with respect to $B$}. We say that {\it $(R,\Delta)$ is big Cohen-Macaulay regular with respect to $B$} (or simply $\text{BCM}_{B}$ regular) if $\tau_B(R,\Delta)=R$.
 \end{defn}
\begin{prop}[{\cite{Ma-Schwede}}]\label{BCM test tight closure}
	Let $(R,\m)$ be a complete normal local domain of characteristic $p>0$, $\Delta\ge 0$ an effective $\Q$-Weil divisor on $\Spec R$ and $B$ a big Cohen-Macaulay $R^+$-algebra. Fix an effective canonical divisor $K_R\ge 0$. Suppose that $K_R+\Delta$ is $\Q$-Cartier. Then
	\[
	\tau_B(R,\Delta)=\tau(R,\Delta).
	\]
\end{prop}
\section{ultraproducts}
\subsection{Basic notions}
 In this subsection, we quickly review basic notions from the theory of ultraproduct. The reader is referred to \cite{affine}, \cite{use of ultraproducts} for details.
We fix an infinite set $W$. We use $\mathcal{P}(W)$ to denote the power set of $W$.
\begin{defn}	A nonempty subset $\mathcal{F} \subseteq \mathcal{P}(W)$ is called a {\it filter} if the following two conditions hold.
	\begin{enumerate}[(i)]
	\item If $A, B \in \mathcal{F}$, then $A \cap B\in \mathcal{F}$.
	\item If $A \in \mathcal{F}$ and $A \subseteq B \subseteq W$, then $B \in \mathcal{F}$.
	\end{enumerate}
\end{defn}
\begin{defn}
	Let $\mathcal{F}$ be a filter on $W$.
	\begin{enumerate}
	\item $\mathcal{F}$ is called an {\it ultrafilter} if for all $A \in \mathcal{P}(W)$, we have $A \in \mathcal{F}$ or $A^c \in \mathcal{F}$, where $A^c$ is the complement of $A$.
	\item $\mathcal{F}$ is called {\it principal} if there exists a finite subset $A\subseteq W$ such that $A \in \mathcal{F}$.
	\end{enumerate}
\end{defn}
\begin{rem}
	By Zorn's lemma, non-principal ultrafilters always exist.
\end{rem}
\begin{rem}
	Ultrafilters are an equivalent notion to two-valued finitely additive measures. If we have an ultrafilter $\mathcal{F}$ on $W$, then
	\[
	m(A):=
	\left\{
	\begin{array}{ll}
		1 & (A\in \mathcal{F}) \\
		0 & (A \notin \mathcal{F})
	\end{array}
	\right.
	\]
	is a two-valued finitely additive measure. Conversely, if $m:\mathcal{P}(W)\to \{0,1\}$ is a nonzero finitely additive measure, then $\mathcal{F}:=\{A\subseteq W|m(A)=1\}$ is an ultrafilter. Here $\mathcal{F}$ is principal if and only if there exists an element $w_0$ of $W$ such that $m(\{w_0\})=1$. Hence, $\mathcal{F}$ is not principal if and only if $m(A)=0$ for any finite subset $A$ of $W$.
\end{rem}
\begin{defn}
	Let $A_w$ be a family of sets indexed by $W$ and $\mathcal{F}$ be an ultrafilter on $W$. Suppose that $a_w\in A_w$ for all $w\in W$ and $\varphi$ is a predicate. We say $\phi(a_w)$ holds {\it for almost all $w$} if $\{w\in W|\varphi(a_w) \text{ holds}\}\in \mathcal{F}$.
\end{defn}
\begin{rem}
	This is an analog of ``almost everywhere'' or ``almost surely'' in analysis. The difference is that $m$ is not countably but finitely additive. We can consider elements in $\mathcal{F}$ as ``large'' sets and elements in the complement $\mathcal{F}^c$ as ``small'' sets. If $\mathcal{F}$ is not principal, all finite subsets of $W$ are ``small''.
\end{rem}
\begin{defn}
	Let $A_w$ be a family of sets indexed by $W$ and $\mathcal{F}$ be a non-principal ultrafilter on $W$. The {\it ultraproduct of $A_w$} is defined by 
	\begin{equation*}
		\ulim_w A_w = A_{\infty} := \prod_w A_w/\sim,
	\end{equation*}
	where $(a_w)\sim (b_w)$ if and only if $\{w\in W|a_w=b_w\}\in \mathcal{F}$.
 	We denote the equivalence class of $(a_w)$ by $\ulim_w a_w$.
 \end{defn}
\begin{rem}[{\cite[Section 3]{Lyu}}]
	If $A_w$ are local rings, then the ultraproduct is equivalent to the localization of $\prod A_w$ at a maximal ideal.
\end{rem}
\begin{eg}
	We use $^*\N$ and $^*\R$ to denote the ultraproduct of $|W|$ copies of $\N$ and $\R$ respectively. $^*\N$ is a semiring and $^*\R$ is a field, see Definition-Proposition \ref{ultraprocut of maps}, Theorem \ref{Los's theorem}. $^*\N$ is a non-standard model of Peano arithmetic. $^*\R$ is a system of hyperreal numbers used in non-standard analysis.
\end{eg}
 \begin{defn-prop}\label{ultraprocut of maps}
 	Let $A_{1w},\dots, A_{nw}$, $B_w$ be families of sets indexed by $W$ and $\mathcal{F}$ be a non-principal ultrafilter. Suppose that $f_w:A_{1w}\times \dots \times A_{nw}\to B_w$ is a family of maps. Then we define the {\it ultraproduct $f_\infty = \ulim_w  f_w : A_{1\infty}\times\dots \times A_{n\infty}\to B_\infty$ of $f_w$} by
 	\[
 	f_\infty(\ulim_w a_{1w},\dots, \ulim_w a_{nw}):=\ulim_w f_w(a_{1w},\dots,a_{nw}).
 	\]
 	This is well-defined.
 \end{defn-prop}
\begin{cor}
	Let $A_w$ be a family of rings. Suppose that $B_w$ is an  $A_w$-algebra and $M_w$ is an $A_w$-module for almost all $w$. Then the following hold:
	\begin{enumerate}
		\item $A_\infty$ is a ring.
		\item $B_\infty$ is an $A_\infty$-algebra.
		\item $M_\infty$ is an $A_\infty$-module.
	\end{enumerate} 
\end{cor}
\begin{proof}
	Let $0:=\ulim_w 0$, $1:=\ulim_w 1$ in $A_\infty$, $B_\infty$ and $0:=\ulim_w 0$ in $M_\infty$. By the above Definition-Proposition, $A_\infty$, $B_\infty$ have natural additions, subtractions and multiplications and we have a natural ring homomorphism $A_\infty \to B_\infty$. Similarly, $M_\infty$ has a natural addition and a scalar multiplication between elements of $M_\infty$ and $A_\infty$.
\end{proof}
\begin{prop}
	Suppose that, for almost all $w$, we have an exact sequence
	\[
	0\to L_w\to M_w \to N_w\to 0
	\]
	of abelian groups. Then
	\[
	0\to \ulim_w L_w \to \ulim_w M_w \to \ulim_w N_w \to 0
	\]
	is an exact sequence of abelian groups. In particular, $\ulim_w:\prod_w\operatorname{Ab}\to \operatorname{Ab}$ is an exact functor.
\end{prop}
\begin{proof}
	Let $f_w:L_w\to M_w$ and $g_w:M_w\to N_w$ be the morphisms in the given exact sequence. Here we only prove the injectivity of $\ulim_w f_w$ and the surjectivity of $\ulim_w g_w$. Suppose that $\ulim_w f_w(a_w)=0$ for $\ulim_w a_w \in \ulim_w L_w$. Then $f_w(a_w)=0$ for almost all $w$. Since $f_w$ is injective for almost all $w$, we have $a_w=0$ for almost all $w$. Therefore, $\ulim_w a_w=0$ in $\ulim_w L_w$. Hence, $\ulim_w f_w$ is injective. Next, let $\ulim_w c_w$ be any element in $\ulim_w N_w$. Since $g_w$ is surjective for almost all $w$, there exists $b_w\in M_w$ such that $g_w(b_w)=c_w$ for almost all $w$. Let $b=\ulim_w b_w$. Then we have $(\ulim_w g_w)(b)=\ulim_w g_w(b_w)=\ulim_w c_w$. Hence, $\ulim_w g_w$ is surjective. The rest of the proof is similar.
\end{proof}

 \L o\'{s}'s theorem is a fundamental theorem in the theory of ultraproducts. We will prepare some notions needed to state the theorem.
\begin{defn}
	The {\it language $\mathcal{L}$ of rings} is the set defined by
	\[
	\mathcal{L}:=\{0,1,+,-,\cdot \}.
	\]
\end{defn}
\begin{defn}
	{\it Terms of $\mathcal{L}$} are defined as follows:
	\begin{enumerate}[(i)]
		\item $0$,$1$ are terms.
		\item Variables are terms.
		\item If $s$, $t$ are terms, then ${-(s)}, (s)+(t), (s)\cdot(t)$ are terms.
		\item  A string of symbols is a term only if it can be shown to be a term by finitely many applications of the above three rules.
	\end{enumerate}
	 We omit parentheses and ``$\cdot$'' if there is no ambiguity.
\end{defn}
\begin{eg}
	$1+1$, $x_1(x_2+1)$,$-(-x)$ are terms.
\end{eg}
\begin{defn}
	{\it Formulas of $\mathcal{L}$} are defined as follows:
	\begin{enumerate}[(i)]
		\item If $s$, $t$ are terms, then $(s=t)$ is a formula.
		\item If $\varphi,\psi$ are formulas, then $(\varphi \land \psi), (\varphi \lor \psi), (\varphi \to \psi),(\lnot \varphi)$ are formulas.
		\item If $\varphi$ is a formula and $x$ is a variable, then $\forall x \varphi, \exists x \varphi$ are formulas.
		\item A string of symbols is a formula only if it can be shown to be a formula by finitely many applications of the above three rules.
	\end{enumerate}
	We omit parentheses if there is no ambiguity and use $\neq$, $\nexists$ in the usual way.
\end{defn}
\begin{rem}
	$\varphi \land \psi$ means ``$\varphi$ and $\psi$,'' $\varphi \lor \psi$ means ``$\varphi$ or $\psi$,'' $\varphi \to \psi$ means ``$\varphi$ implies $\psi$'' and $\lnot \varphi$ means ``$\varphi$ does not hold.''
\end{rem}
\begin{eg}
	0=1, $x=0 \land y\neq 1$, $\forall x \forall y (xy=yx)$ are formulas.
\end{eg}
\begin{rem}
	Variables in a formula $\varphi$ which is not bounded by $\forall$ or  $\exists$ are called free variables of $\varphi$. If $x_1,\dots,x_n$ are free variables of $\varphi$, we denote $\varphi (x_1,\dots,x_n)$ and we can substitute elements of a ring for $x_1,\dots,x_n$.
\end{rem}
\begin{thm}[\L o\'{s}'s theorem in the case of rings]\label{Los's theorem}
	Suppose that $\varphi(x_1,\dots,x_n$) is a formula of $\mathcal{L}$ and $A_w$ is a family of rings indexed by a set $W$ endowed with a non-principal ultrafilter. Let $a_{iw}\in A_w$. Then $\varphi (\ulim_w a_{1w}, \dots, \ulim_w a_{nw})$ holds in $A_\infty$ if and only if $\varphi(a_{1w},\dots,a_{nw})$ holds in $A_w$ for almost all $w$.
\end{thm}
\begin{rem}
	Even if $A_w$ are not rings, replacing $\mathcal{L}$ properly, we can get the same theorem as above. We use one in the case of modules.
\end{rem}
\begin{eg}\label{formulas example}
	Let $A$ be a ring. If a property of rings is written by some formula, we can apply \L o\'{s}'s theorem.
	\begin{enumerate}
		\item $A$ is a field if and only if $\forall x (x=0 \lor \exists y(xy=1))$ holds.
		\item $A$ is a domain if and only if $\forall x \forall y(xy=0\to(x=0\lor y=0))$ holds.
		\item $A$ is a local ring if and only if 
		\[
		\forall x \forall y (\nexists z (xz=1)\land \nexists w(yw=1)\to \nexists u((x+y)u=1))
		\]
		holds.
		\item The condition that $A$ is an algebraically closed field is written by countably many formulas, i.e., the formula in (1) and for all $n\in \N$, 
		\[\forall a_0 \dots a_{n-1} \exists x (x^n+a_{n-1}x^{n-1}+\dots +a_0=0).\]
		\item The condition that $A$ is Noetherian cannot be written by formulas. Indeed, if $W=\N$ with some non-principal ultrafilter and $A_w=\C \llbracket x \rrbracket $, then $\ulim_n x^n\neq 0$ is in $\cap_n \m_\infty^n$, where $\m_\infty$ is the maximal ideal of $A_\infty$. Hence, $A_\infty$ is not Noetherian.
	\end{enumerate}
\end{eg}
\begin{prop}[{\cite[2.8.2]{affine}, see Example \ref{formulas example}}]
	If almost all $K_w$ are algebraically closed field, then $K_\infty$ is an algebraically closed field.
\end{prop}
\begin{thm}[Lefschetz principle, {\cite[Theorem 2.4]{affine}}]
	Let $W$ be the set of prime numbers endowed with some non-principal ultrafilter. Then
	\begin{equation*}
		\ulim_{p\in\mathcal{W}} \bar{\F_p}\cong \mathbb{C}.
	\end{equation*}
\end{thm}
\begin{proof}
	Let $C=\ulim_p\bar{\F_p}$. By the above theorem, $C$ is an algebraically closed field. For any prime number $q$, we have $q\neq 0$ in $\bar{\F_p}$ for almost all $p$. Hence, $q\neq 0$ in $C$, i.e., $C$ is of characteristic zero. We can check that $C$ has the same cardinality as $\C$. If two algebraically closed uncountable field of characteristic zero have the equal cardinality, then they are isomorphic. Hence, $C\cong \C$. (Note that this isomorphism is not canonical).
\end{proof}
\subsection{Non-standard hulls}
 In this subsection, we will introduce the notion of non-standard hulls along \cite{affine}, \cite{use of ultraproducts}. Throughout this subsection, let $\mathcal{P}$ be the set of prime numbers and we fix a non-principal ultrafilter on $\mathcal{P}$ and an isomorphism $\ulim_p \bar{\F_p}\cong \C$.
 
  Let $\C[X_1,\dots,X_n]_\infty:=\ulim_p\bar{\F_p}[X_1,\dots,X_n]$. Then we have the following proposition.
 \begin{prop}[{\cite[Theorem 2.6]{affine}}]
 	We have a natural map $\C[X_1,\dots,X_n]\to \C[X_1,\dots,X_n]_\infty$, which is faithfully flat.
 \end{prop}
\begin{defn}
	The ring $\C[X_1,\dots,X_n]_\infty$ is said to be the {\it non-standard hull of $\C[X_1,\dots,X_n]$}.
\end{defn}
\begin{rem}
	If $n\ge 1$, then $\C[X_1,\dots,X_n]_\infty$ is not Noetherian. Let $y=\ulim_p X_1^p$. Then, for any integer $l\ge 1$, $X_1^p\in (X_1,\dots,X_n)^l$ for almost all $p$. Hence, $y\in (X_1,\dots,X_n)^l$ for any $l$ by \L o\'{s}'s theorem. Therefore, $\cap_l (X_1,\dots,X_n)^l\neq 0$. By Krull's intersection theorem, $\C[X_1,\dots,X_n]_\infty$ is not Noetherian.
\end{rem}
\begin{defn}
	Suppose that $R$ is a finitely generated $\C$-algebra. Let 
	\[
	R\cong \C[X_1,\dots,X_n]/I
	\] be a presentation of $R$. The {\it non-standard hull $R_\infty$ of $R$} is defined by
	\[
	R_\infty:=\C[X_1,\dots,X_n]_\infty/I\C[X_1,\dots,X_n]_\infty.
	\]
\end{defn}
\begin{rem}
	The non-standard hull is independent of a representation of $R$. If $R\cong \C[X_1,\dots,X_n]/I\cong \C[Y_1,\dots,Y_m]/J$, then $\bar{\F_p}[X_1,\dots,X_n]/I_p \cong \bar{\F_p}[Y_1,\dots,Y_m]/J_p$ for almost all $p$, see Definition \ref{approximations}, Definition \ref{approximation of homomorphism}.

\end{rem}
\begin{rem}
	The natural map $R\to R_\infty$ is faithfully flat since this is a base change of the homomorphism $\C[X_1,\dots,X_n]\to \C[X_1,\dots,X_n]_\infty$. By faithfully flatness, we have $IR_\infty\cap R=R$ for any ideal $I\subseteq R$.
\end{rem}
\begin{defn}
	Let $a\in \C$. Since $\ulim_p \bar{\mathbb{F}_p}\cong \C$, we have a family $(a_p)_p$ of elements of $\bar{\mathbb{F}_p}$ such that $\ulim a_p=a$. Then we call $(a_p)_p$ an {\it approximation of $a$}.
\end{defn}
\begin{prop}\label{affine non-standard hull prop}
	Let $I=(f_1,\dots,f_s)$ be an ideal of $\C[X_1,\dots, X_n]$ and $f_i=\sum a_{i\nu}X^\nu$. Let $I_p=(f_{1p},\dots,f_{sp})\bar{\F_p}[X_1,\dots,X_n]$, where $f_{ip}=\sum a_{i\nu p}X^\nu$ and each $(a_{i\nu p})_p$ is an approximation of $a_{i\nu}$. Then we have
	\[
	I\C[X_1,\dots,X_n]_\infty =\ulim_p I_p
	\]
	and 
	\[
	R_\infty\cong \ulim_p (\bar{\F_p}[X_1,\dots,X_n]/I_p).
	\]
\end{prop}
\begin{defn}\label{approximations}
	Let $R$ be a finitely generated $\C$-algebra. 
	\begin{enumerate}
		\item In the setting of Proposition \ref{affine non-standard hull prop},  a family $R_p$ is said to be an {\it approximation of $R$} if $R_p$ is an $\bar{\F_p}$-algebra and $R_p \cong \bar{\F_p}[X_1,\dots,X_n]/I_p$ for almost all $p$. Then we have $R_\infty  \cong \ulim_p R_p$.
		\item For an element $f\in R$, a family $f_p$ is said to be an {\it approximation of $f$} if $f_p\in R_p$ for almost all $p$ and $f=\ulim_p f_p$ in $R_\infty$. For $f\in R_\infty$, we define an {\it approximation of $f$} in the same way.
		\item For an ideal $I=(f_1,\dots,f_s) \subseteq R$, a family $I_p$ is said to be an {\it approximation of $I$} if $I_p$ is an ideal of $R_p$ and $I_p=(f_{1p},\dots,f_{sp})$ for almost all $p$. For finitely generated ideal $I\subseteq R_\infty$, we define {\it an approximation of $I$} in the same way.
	\end{enumerate}
\end{defn}
\begin{rem}
	This is an abuse of notation since approximations should be denoted by $(R_p)_p$, $(f_p)_p$, $(I_p)_p$, etc.
\end{rem}
\begin{defn}\label{approximation of homomorphism}
	Let $\phi:R\to S$ be a $\C$-algebra homomorphism between finitely generated $\C$-algebras. 
	Suppose that $R\cong \C[X_1,\dots,X_n]/I$ and $S\cong\C[Y_1,\dots,Y_m]/J$. Let $f_i\in \C[Y_1,\dots,Y_m]$ be a lifting of the image of $X_i \mod I$ under $\phi$.
	Then we define an {\it approximation $\varphi_p:R_p\to S_p$ of $\varphi$} as the morphism induced by $X_i\mapsto f_{ip}$.
	Let $\varphi_\infty:=\ulim_p\varphi_p$, then the following diagram commutes.
	\[
	\xymatrix{
	R \ar[r]^{\varphi} \ar[d] & S \ar[d] \\
	R_\infty \ar[r]^{\varphi_\infty} & S_\infty
	}
	\]
\end{defn}

\begin{prop}[{\cite[Corollary 4.2]{affine},\cite[Theorem 4.3.4]{use of ultraproducts}}]
	Let $R$ be a finitely generated $\C$-algebra. An ideal $I\subseteq R$ is prime if and only if $I_p$ is prime for almost all $p$ if and only if $IR_\infty$ is prime.
\end{prop}

\begin{defn}
	Let $R$ be a local ring essentially of finite type over $\C$. Suppose that $R\cong S_\p$, where $S$ is a finitely generated $\C$-algebra and $\p$ is a prime ideal of $S$. Then we define the {\it non-standard hull $R_\infty$ of $R$} by
	\[
	R_\infty:=(S_\infty)_{\p S_\infty}.
	\]
\end{defn}
\begin{rem}
	Since $S\to S_\infty$ is faithfully flat, $R\to R_\infty$ is faithfully flat.
\end{rem}
\begin{defn}
		Let $S$ be a finitely generated $\C$-algebra, $\p$ a prime ideal of $S$ and $R\cong S_\p$.
	\begin{enumerate}
		\item A family $R_p$ is said to be an {\it approximation of $R$} if $R_p$ is an $\bar{\F_p}$-algebra and $R_p \cong (S_p)_{\p_p}$ for almost all $p$. Then we have $R_\infty  \cong \ulim_p R_p$.
		\item For an element $f\in R$, a family $f_p$ is said to be an {\it approximation of $f$} if $f_p\in R_p$ for almost all $p$ and $f=\ulim_p f_p$ in $R_\infty$. For $f\in R_\infty$, we define an {\it approximation of $f$} in the same way.
		\item For an ideal $I=(f_1,\dots,f_s) \subseteq R$, a family $I_p$ is said to be an {\it approximation of $I$} if $I_p$ is an ideal of $R_p$ and $I_p=(f_{1p},\dots,f_{sp})$ for almost all $p$. For finitely generated ideal $I\subseteq R_\infty$, we define {\it an approximation of $I$} in the same way.
	\end{enumerate}
\end{defn}
\begin{defn}
	Let $S_1,S_2$ be finitely generated $\C$-algebras and $\p_1,\p_2$ prime ideals of $S_1,S_2$ respectively. Suppose that $R_i\cong (S_i)_{\p_i}$ and $\varphi:R_1\to R_2$ is a local $\C$-algebra homomorphism. Let $S_1\cong \C[X_1,\dots,X_n]/I$ and $f_j/g_j$ be the image of $X_j$ under $\varphi$, where $f_j\in S_2$, $g_j\in S_2\setminus \p_2$. Then we say that a homomorphism $R_{1p}\to R_{2p}$ induced by $X_j\mapsto f_{jp}/g_{jp}$ is an {\it approximation of $\varphi$}. Let $\varphi_\infty:=\ulim_p \varphi_p$. Then the following commutative diagram commutes:
	\[
	\xymatrix{
		R \ar[r]^{\varphi} \ar[d] & S \ar[d] \\
		R_\infty \ar[r]^{\varphi_\infty} & S_\infty
	}.
	\]
\end{defn}
\begin{defn}
	Let $R$ be a finitely generated $\C$-algebra or a local ring essentially of finite type over $\C$ and let $M$ be a finitely generated $R$-module. Write $M$ as the cokernel of a matrix $A$, i.e., given by an exact sequence
	\[
	R^m\xrightarrow{A} R^n\to M \to 0,
	\]
	where $m,n$ are positive integers. Let $A_p$ be an approximation of $A$ defined by entrywise approximations. Then the cokernel $M_p$ of the matrix $A_p$ is called an {\it approximation of $M$} and the ultraproduct $M_\infty:=\ulim_p M_p$ is called the {\it non-standard hull of $M$}. $M_\infty$ is a finitely generated $R_\infty$-module and independent of the choice of matrix $A$.
\end{defn}
\begin{rem}
	Tensoring the above exact sequence with $R_\infty$, we have an exact sequence
	\[
	R_\infty^m \xrightarrow{A} R_\infty^n \to M\otimes_R R_\infty \to 0.
	\]
	Taking the ultraproduct of exact sequences
	\[
	R_p^m\xrightarrow{A_p} R_p^n \to M_p \to 0,
	\]
	we have an exact sequence
	\[
	R_\infty^m\xrightarrow{A} R_\infty^n \to M_\infty\to 0.
	\]
	Therefore, $M_\infty\cong M\otimes_R R_\infty$.
	Note that if $m,n$ is not integers but infinite cardinals, then the naive definition of an approximation of $A$ does not work and the ultraproduct of $R_p^{\oplus n}$ is not necessarily equal to $R_\infty^{\oplus n}$.
\end{rem}
Here we state basic properties about non-standard hulls and approximations.
\begin{prop}[{\cite[2.9.5, 2.9.7, Theorem 4.5, Theorem 4.6]{affine},\cite[Section 4.3]{use of ultraproducts}, cf. \cite[5.1]{AS}}]
	Let $R$ be a local ring esseentially of finite type over $\C$, then the following hold:
	\begin{enumerate}
		\item $R$ has dimension $d$ if and only if $R_p$ has dimension $d$ for almost all $p$.
		\item $\x=x_1,\dots,x_i$ is an $R$-regular sequence if and only if $\x_p=x_{1p},\dots, x_{ip}$ is an $R_p$-regular sequence for almost all $p$ if and only if $\x$ is an $R_\infty$-regular sequence.
		\item $\x=x_1,\dots,x_d$ is a system of parameters of $R$ if and only if $\x_p$ is a system of parameters of $R_p$ for almost all $p$.
		\item $R$ is regular if and only if $R_p$ is regular for almost all $p$.
		\item $R$ is Gorenstein if and only if $R_p$ is Gorenstein for almost all $p$.
		\item $R$ is Cohen-Macaulay if and only if $R_p$ is Cohen-Macaulay for almost all $p$.
	 	\end{enumerate}
\end{prop}
\begin{prop}[{\cite[Proposition 3.9]{TY}}]
	Let $R$ be a local ring essentially of finite type over $\mathbb{C}$. The  following conditions are equivalent to each other.
	\begin{enumerate}
		\item $R$ is normal.
		\item $R_p$ is normal for almost all $p$.
		\item $R_\infty$ is normal.
	\end{enumerate}
\end{prop}
\begin{defn}
	Let $R$ be a normal local domain essentially of finite type over $\C$ and $\Delta=\sum_i a_i\Delta_i$ a $\Q$-Weil divisor. Assume that $\Delta_i$ are prime divisors and $\p_i$ is a prime ideal associated to $\Delta_i$ for each $i$. Suppose that $\p_{ip}$ is an approximation of $\p_{i}$ and $\Delta_{ip}$ is a divisor associated to $\p_{ip}$. We say $\Delta_p:=\sum_i a_i\Delta_{ip}$ is an {\it approximation of $\Delta$}.
\end{defn}
\begin{rem}
	If $\Delta$ is an effective integral divisor, then this definition is compatible with Definition \ref{approximations} by \cite[Theorem 4.4]{affine}. Hence, if $\Delta$ is $\Q$-Cartier, then $\Delta_p$ is $\Q$-Cartier for almost all $p$.
\end{rem}
Lastly, we review some singularities introduced by Schoutens via ultraproducts.
\begin{defn}[{\cite[Definition 5.2]{affine},\cite[Definition 3.1]{Puresubrings}}]
	Suppose that $R$ is a finitely generated $\C$-algebra or a local domain essentially of finite type over $\C$. Let $I\subseteq R$ be an ideal. The {\it generic tight closure $I^{*\gen}$ of $I$} is defined by
	\[
	I^{*\gen}=(\ulim_p I_p)^*\cap R.
	\] 
\end{defn}
\begin{rem}
	The generic tight closure $I^{*\gen}$ of $I$ does not depend on the choice of approximation of $I$ since any two approximations are almost equal.
\end{rem}
\begin{defn}[{\cite[Definition 4.1, Remark 4.7]{Puresubrings},\cite[Definition 4.3]{canonical BCM}}]\label{def of generically}
	Suppose that $R$ is a finitely generated $\C$-algebra or a local ring essentially of finite type over $\C$.
	\begin{enumerate}
		\item $R$ is said to be {\it weakly generically $F$-regular} if $I^{*\gen}=I$ for any ideal $I\subseteq R$.
		\item $R$ is said to be {\it generically $F$-regular} if $R_\p$ is weakly generically $F$-regular for any prime ideal $\p\in \Spec R$.
		\item Let $R$ be a local ring essentially of finite type over $\C$. $R$ is said to be {\it generically $F$-rational} if $I^{*\gen}=I$ for some ideal $I$ generated by a system of parameters.
	\end{enumerate}
\end{defn}
\begin{prop}[{\cite[Theorem 4.3]{Puresubrings}}]
	If $R$ is generically $F$-rational, then $I^{*\gen}=I$ for any ideal $I$ generated by part of a system of parameters.
\end{prop}
\begin{prop}[{\cite[Theorem 6.2]{Puresubrings},\cite[Proposition 4.5, Theorem 4.12]{canonical BCM}}]\label{generically F-rational rational}
	If $R$ is generically $F$-rational if and only if $R_p$ is $F$-rational for almost all $p$ if and only if $R$ has rational singularities.
\end{prop}
\begin{defn}[{\cite[3.2]{log-terminal}}]
	Let $R$ be a local ring essentially of finite type over $\C$ and $R_p$ be an approximation.
	Let $\epsilon:=\ulim_p e_p \in {^*\N}$. Then an {\it ultra-Frobenius $F^\epsilon:R\to R_\infty$ associated to $\epsilon$} is defined by $x\mapsto \ulim_p (F_p^{e_p}(x_p))$, where $F_p$ is a Frobenius morphism in characteristic $p$.
\end{defn}
\begin{defn}[{\cite[Definition 3.3]{log-terminal}}]\label{def of ultra-F-regular}
	Let $R$ be a local domain essentially of finite type over $\C$. $R$ is said to be {\it ultra-$F$-regular} if, for each $c\in R^{\circ}$, there exists $\epsilon \in {^*\N}$ such that 
	\[
	R \xrightarrow{cF^\epsilon} R_\infty
	\]
	is pure.
\end{defn}
\begin{prop}[{\cite[Theorem A]{log-terminal}}]\label{ultra-F-regular log-terminal}
	Let $R$ be a  $\Q$-Gorenstein normal local domain essentially of finite type over $\C$. Then $R$ is ultra-$F$-regular if and only if $R$ has log-terminal singularities.
\end{prop}
\subsection {Relative hulls}
In this subsection we introduce the concept of relative hulls and approximations of schemes, cohomologies, etc. We refer the reader to \cite{affine}, \cite{log-terminal}, \cite{Puresubrings}.
\begin{defn}[cf. \cite{Puresubrings}]
	Let $R$ be a local ring essentially of finite type over $\C$. Suppose that $X$ is a finite tuple of indeterminates and $f\in R[X]$ is a polynomial such that $f=\sum_{\nu} a_\nu X^\nu$, where $\nu$ is a multi-index. If $a_{\nu p}$ is an approximation of $a_{\nu}$ for each $\nu$, then the sequence of polynomials $f_p:=\sum_\nu a_{\nu p} X^\nu$ is said to be an {\it $R$-approximation of $f$}.
	If $I:=(f_1,\dots,f_s)$ is an ideal in $R[X]$, then we call $I_{p}:=(f_{1p},\dots,f_{sp})R_p[X]$ an {\it $R$-approximation} of $I$, and if $S=R[X]/I$, then we call $S_p:=R_p[X]/I_p$ an {\it $R$-approximation} of $S$.
\end{defn}
\begin{rem}
	Any two $R$-approximations of a polynomial $f$ are almost equal. Similarly, any two $R$-approximations of an ideal $I$ are almost equal.
\end{rem}
\begin{defn}[cf. \cite{Puresubrings}]
	Let $S$ be a finitely generated $R$-algebra and $S_p$ an $R$-approximation of $S$, then we call $S_\infty=\ulim _p S_p$ the {\it (relative) $R$-hull of $S$}.
\end{defn}
\begin{defn}[cf. \cite{log-terminal}]
	If $X$ is an affine scheme $\Spec S$ of finite type over $\Spec R$, then we call $X_p:=\Spec S_p$ is an {\it $R$-approximation of $X$}.
\end{defn}
\begin{defn}[cf. \cite{log-terminal}]
	Suppose that $f:Y\to X$ is a morphism of affine schemes of finite type over $\Spec R$. If $X=\Spec S, Y=\Spec T$ and $\varphi:S\to T$ is the morphism corresponding to $f$, then we call $f_p:Y_p\to X_p$ is an {\it $R$-approximation of $f$}, where $f_p$ is a morphism of $R_p$-schemes induced by an $R$-approximation $\varphi_p:S_p\to T_p$.
\end{defn}
\begin{defn}[cf. \cite{log-terminal}]
	Let $S$ be a finitely generated $R$-algebra and $M$ a finitely generated $S$-module. Write $M$ as the cokernel of a matrix $A$, i.e., given by an exact sequence
	\[
	S^m\xrightarrow{A} S^n\to M\to 0,
	\]
	where $m,n$ are positive integers.
	Let $A_p$ be an $R$-approximation of $A$ defined by entrywise $R$-approximations. Then the cokernel $M_p$ of the matrix $A_p$ is called an {\it $R$-approximation of $M$} and the ultraproduct $M_\infty :=\ulim_p M_p$ is called the {\it $R$-hull of $M$}. $M_\infty$ is independent of the choice of the matrix $A$ and $M_\infty\cong M\otimes_S S_\infty$. 
\end{defn}
\begin{rem}
	If $M$ is not finitely generated, then we cannot define an $R$-approximation of $M$ in this way. It is crucial that any two $R$-approximations of $A$ is equal for almost all $p$.
\end{rem}
\begin{defn}[\cite{log-terminal}]
	Let $X$ be a scheme of finite type over $\Spec R$. Let $\mathfrak{U}=\{U_i\}$ is a finite affine open covering of $X$ and $U_{ip}$ be an $R$-approximation of $U_i$. Gluing $\{U_{ip}\}$ together, we obtain a scheme $X_p$ of finite type over $\Spec R_p$. We call $X_p$ an {\it $R$-approximation of $X$}.
\end{defn}
\begin{rem}
	Suppose that $\{U_{ijk}\}_k$ is a finite affine open covering of $U_i\cap U_j$ and $\phi_{ijk}:\sO_{U_i}|_{U_k}\cong \sO_{U_j}|_{U_k}$ are isomorphisms. Then $R$-approximations $\phi_p:\sO_{U_{ip}}|_{U_{kp}}\to \sO_{U_{jp}}|_{U_{kp}}$ are isomorphisms for almost all $p$ (note that indices $ijk$ are finitely many). Hence, we can glue these together. For any other choice of finite affine open covering $\mathfrak{U}'$ of $X$ , the resulting $R$-approximation $X'_p$ is isomorphic to $X_p$ for almost all $p$.
\end{rem}
\begin{defn}[cf. \cite{log-terminal}]
	Suppose that $f:Y\to X$ is a morphism between schemes of finite type over $\Spec R$. Let $\mathfrak{U}$, $\mathfrak{V}$ be finite affine open coverings of $X$ and $Y$ respectively such that for any $V\in \mathfrak{V}$, there exists some $U\in\mathfrak{U}$ such that $f(V)\subseteq U$. Let $\mathfrak{U}_p$, $\mathfrak{V}_p$ be $R$-approximations of $\mathfrak{U}$, $\mathfrak{V}$ and $(f|_V)_p$ an {\it $R$-approximation of $f|_V$}. We define an $R$-approximation $f_p$ of $f$ by the morphism determined by $(f|V)_p$.
\end{defn}
\begin{rem}
	In the same way as the above Remark, $(f|_V)_p$ and $(f|_{V'})_p$ agree on $V\cap V'$ for any two opens $V,V'\in \mathfrak{V}$ for almost all $p$.
\end{rem}
\begin{defn}[cf. \cite{log-terminal}]
	Let $X$ be a scheme of finite type over $\Spec R$ and $\mathcal{F}$ a coherent $\sO_X$-module. Let $\mathfrak{U}$ be a finite affine open covering of $X$. For any $U\in \mathfrak{U}$, we have an $R$-approximation $M_{Up}$ of $M_{U}$ such that $M_{U}$ is a finitely generated $\sO_U$-module and $\tilde{M_U}\cong \mathcal{F}|_U$. We define an {\it $R$ -approximation $\mathcal{F}_p$ of $\mathcal{F}$} by the coherent $\sO_{X_p}$-module determined by $\tilde{M_{Up}}$.
\end{defn}
\begin{defn}[cf. \cite{log-terminal}]
	Let $X$ be a separated scheme of finite type over $\Spec R$ and $\mathcal{F}$ a coherent $\sO_X$-module. Then the ultra-cohomology of $\mathcal{F}$ is defined by
	 \[
	 	H_\infty^i(X,\mathcal{F}):=\ulim_p H^i(X_p,\mathcal{F}_p).
	 \]
\end{defn}
\begin{rem}
	In the above setting, let $\mathfrak{U}=\{U_i\}_{i=1,\dots,n}$ be a finite affine open covering of $X$, let 
	\[
	C^j(\mathfrak{U}, \mathcal{F}):=\prod_{i_0<\dots<i_j}\mathcal{F}(U_{i_0\dots i_j}),
	\]
	where $U_{i_0\dots i_j}:=U_{i_0}\cap \dots \cap U_{i_j}$, and let
	\[
	(C^j(\mathfrak{U},\mathcal{F}))_p:=\prod_{i_0\dots i_j}(\mathcal{F}(U_{i_0\dots i_j}))_p,
	\]
	where $\mathcal{F}(U_{i_0\dots i_j})_p$ is an $R$-approximation considered as $\sO(U_{i_0\dots i_j})$-module. Then
	\[
	(C^j(\mathfrak{U},\mathcal{F}))p
	\]
	coincides with the $j$-th term of the \v{C}ech complex of $X_p$, $\mathfrak{U}_p$ and $\mathcal{F}_p$. We have a commutative diagram
	\[
	\xymatrix{
	C^{j-1}(\mathfrak{U},\mathcal{F}) \ar[r] \ar[d] & C^j(\mathfrak{U},\mathcal{F}) \ar[r] \ar[d] & C^{j+1}(\mathfrak{U},\mathcal{F}) \ar[d] \\
	\ulim_p (C^{j-1}(\mathfrak{U},\mathcal{F}))_p \ar[r] & \ulim_p (C^j(\mathfrak{U},\mathcal{F}))_p \ar[r] & \ulim_p(C^{j+1}(\mathfrak{U},\mathcal{F}))_p.
	}
	\]
	Since $\ulim_p (\mathchar`-)$ is an exact functor, we have
	\[
	\check{H}^j(\mathfrak{U},\mathcal{F})\to \ulim_p \check{H}^j(\mathfrak{U}_p,\mathcal{F}_p).
	\]
	If $X$ is separated, then $X_p$ is separated for almost all $p$. This can be checked by taking a finite affine open covering and observing that if the diagonal morphism $\Delta_{X/\Spec R}$ is a closed immersion, then $\Delta_{X_p/\Spec R_p}$ is also a closed immersion for almost all $p$. Hence, we have the map
	\[
	H^j(\mathfrak{U},\mathcal{F})\to \ulim_p H^j(\mathfrak{U}_p,\mathcal{F}_p).
	\]
	Note that we do not know whether this map is  injective or not.
\end{rem}
\begin{prop}
	Let $R$ be a local ring essentially of finite type over $\C$ of dimension $d$, $\x=x_1,\dots,x_d$ a system of parameters and $M$ a finitely generated $R$-module. Then we have a natural homomorphism $H_\m^d(M)\to \ulim_p H_{\m_p}^d(M_p)$.
\end{prop}
\begin{proof}
	Since $M_{x_1\cdots \hat{x_i}\cdots x_d}$ is a finitely generated $R_{x_1\cdots \hat{x_i}\cdots x_d}$-module and $M_{x_1\cdots x_d}$ is a finitely generated $R_{x_1\cdots x_d}$-module, we have an $R$-approximation $(M_{x_1\cdots\hat{x_1}\cdots x_d})_p\cong (M_p)_{x_{1p}\cdots \hat{x_{ip}}\cdots x_{dp}}$ and $(M_{x_1\cdots x_d})_p\cong (M_p)_{x_{1p}\cdots x_{dp}}$ for almost all $p$. We have a commutative diagram
	\[
	\xymatrix{
		\bigoplus_i M_{x_1\cdots \hat{x_i}\cdots x_d} \ar[r] \ar[d] & M_{x_1\cdots x_d} \ar[d] \\
		\bigoplus_i \ulim_p(M_p)_{x_{1p}\cdots \hat{x_{ip}}\cdots x_{dp}} \ar[r] & \ulim_p (M_p)_{x_{1p}\cdots x_{dp}}
	}.
	\]
	Taking the cokernel of rows, we have the desired map.
\end{proof}
\begin{rem}
	We do not know whether $H_\m^d(M)\to \ulim_p H_{\m_p}^d(M_p)$ is injective or not.
\end{rem}
\begin{prop}
	Let $R$ be a local ring essentially of finite type over $\C$ of dimension $d$, $\x=x_,\dots,x_d$ a system of parameters and $M_p$ an $R_p$-module for almost all $p$. Then we have a natural homomorphism $H_\m^d(\ulim_p M_p)\to \ulim_p H_\m^d(M_p)$.
\end{prop}
\begin{proof}
	We have a commutative diagram
	\[
	\xymatrix{
		\bigoplus_i (\ulim_p M_p)_{x_1\cdots \hat{x_i}\cdots x_d} \ar[r] \ar[d] & (\ulim_p M_p)_{x_1\cdots x_d} \ar[d] \\
		\bigoplus_i \ulim_p(M_p)_{x_{1p}\cdots \hat{x_{ip}}\cdots x_{dp}} \ar[r] & \ulim_p (M_p)_{x_{1p}\cdots x_{dp}}
	}.	
	\]
\end{proof}
Taking the cokernel of rows, we have the desired map.
\section{Big Cohen-Macaulay algebras constructed via ultraproducts}
In \cite{canonical BCM}, Schoutens constructed the canonical big Cohen-Macaulay algebra in characteristic zero. Following the idea of \cite{canonical BCM}, we will deal with big Cohen-Macaulay algebras constructed via ultraproducts in slightly general settings.
In this section, suppose that $(R,\m)$ is a local domain essentially of finite type over $\C$ and $R_p$ is an approximation of $R$. 
\begin{defn}[{\cite[Section 2]{canonical BCM}}]
	Suppose that $R$ is a local domain essentially of finite type over $\C$. Then we define the {\it canonical big Cohen-Macaulay algebra $\mathcal{B}(R)$ of $R$} by
	\[\mathcal{B}(R):=\ulim_p R_p^{+}.
	\]
\end{defn}

\begin{setting}\label{setting BCM}
	Let $R$ be a local domain essentially of finite type over $\C$ of dimension $d$ and let $B_p$ be a big Cohen-Macaulay ${R_p}^+$-algebra for almost all $p$. We use $B$ to denote $\ulim_p B_p$.
\end{setting}
\begin{rem}
	By Theorem \ref{absolute integral closure p>0}, we can set $B_p=R_p^+$ and $B=\mathcal{B}(R)$ in Setting \ref{setting BCM}.
\end{rem}
\begin{prop}
	$\mathcal{B}(R)$ is a domain over $R^+$-algebra.
\end{prop}
\begin{proof}
	By \L o\'{s}'s theorem, $\mathcal{B}(R)$ is a domain over $R_\infty=\ulim_p R_p$. Hence, $\mathcal{B}(R)$ is an $R$-algebra. Let $f=\sum a_{n} x^n\in \mathcal{B}(R)[x]$ be a monic polynomial in one variable over $\mathcal{B}(R)$ and let $f_p=\sum a_{np}x^n$ be an approximation of $f$. Since $f_p$ is a monic polynomial for almost all $p$ and $R_p^+$ is absolutely integrally closed, $f_p$ has a root $c_p$ in $R_p^+$ for almost all $p$. Hence, $c:=\ulim_p c_p\in \mathcal{B}(R)$ is a root of $f$ by \L o\'{s}'s theorem. Hence, $\mathcal{B}(R)$ is absolutely integrally closed. In particular, $\mathcal{B}(R)$ contains an absolute integral closure $R^+$ of $R$.
\end{proof}
\begin{cor}
	In Setting \ref{setting BCM}, $B$ is an $R^+$-algebra.
\end{cor}
\begin{proof}
	Since $B_p$ is an $R_p^+$-algebra for almost all $p$, $B$ is an $R^+$-algebra by the above proposition.
\end{proof}
\begin{prop}
	In Setting \ref{setting BCM}, $B$ is a big Cohen-Macaulay $R$-algebra.
\end{prop}
\begin{proof}
	Assume that $B$ is not a big Cohen-Macaulay $R$-algebra. Since $B_p\neq \m_p B_p$ for almost all $p$, we have $B\neq \m B$. Hence, there exists part of system of parameters $x_1,\dots,x_i$ of $R$ such that $(x_1,\dots,x_{i-1})B\subsetneq (x_1,\dots,x_{i-1})B:_B x_i$. Then there exists $y\in B$ such that $x_i y\in(x_1,\dots,x_{i-1})B$ and $y\notin (x_1,\dots,x_{i-1})B$. Taking approximations, we have $x_{ip}y_p\in (x_{1p},\dots,x_{(i-1)p})B_p$ and $y_p\notin (x_{1p}\dots,x_{(i-1)p})B_p$ for almost all $p$. Since $x_{1p},\dots,x_{ip}$ is part of a system of parameters of $R_p$ and $B_p$ is a big Cohen-Macaulay $R_p$-algebra for almost all $p$, $x_{1p},\dots,x_{ip}$ is a regular sequence for almost all $p$. This is a contradiction. Therefore, $B$ is a big Cohen-Macaulay $R$-algebra.
\end{proof}
\begin{lem}\label{injectivity local cohomology}
	In Setting \ref{setting BCM}, the natural homomorphism $H_\m^d(B)\to \ulim_p H_{\m_p}^d(B_p)$ is injective.
\end{lem}
\begin{proof}
	Let $x=x_1\cdots x_d$ be the product of a system of parameters and $\lbrack \frac{z}{x^t}\rbrack$ be an element of $H_\m^d(B)$ such that the image in $\ulim_p H_{\m_p}^d(B_p)$ is zero. Then there exists $s_p\in\N$ such that $x^{s_p}z\in (x_{1p}^{s_p+t},\dots,x_{dp}^{s_p+t})B_p$ for almost all $p$. Since $B_p$ is a big Cohen-Macaulay $R_p$-algebra for almost all $p$, $z\in (x_{1p}^t,\dots,x_{dp}^t)B_p$ for almost all $p$. Hence, $z\in (x_1^t,\dots,x_d^t)B$ and $\lbrack \frac{z}{x^t} \rbrack=0$ in $H_\m^d(B)$.
\end{proof}
We generalize \cite[Theorem 4.2]{canonical BCM} to the cases other than the canonical big Cohen-Macaulay algebra.
\begin{prop}[{cf. \cite[Theorem 4.2]{canonical BCM}, \cite[Proposition 3.7]{Ma-Schwede}}]\label{rational and BCM rational}
	In Setting \ref{setting BCM}, $R$ is $\text{BCM}_B$-rational if and only if $R$ has rational singularities. In particular, $R$ has rational singularities if $R$ is BCM-rational.
\end{prop}
\begin{proof}
	Let $x:=x_1\cdots x_d$ is the product of a system of parameters. Suppose that $R$ has rational singularities. By \cite[Proposition 4.11]{canonical BCM} and \cite{Hara}, $R_p$ is $F$-rational for almost all $p$. Let $\eta:=\lbrack \frac{z}{x^t} \rbrack$ be an element of $H_\m^d(R)$ such that $\eta =0$ in $H_\m^d(B)$. Then we have a commutative diagram
	\[
	\xymatrix{
		H_\m^d(R) \ar[r] \ar[d] & \ulim_p H_{\m_p}^d(R_p) \ar[d] \\
		H_\m^d(B) \ar[r] & \ulim_p H_{\m_p}^d(B_p).
	}
	\]
	By \cite[Proposition 3.5]{Ma-Schwede}, $H_{\m_p}^d(R_p)\to H_{\m_p}^d(B_p)$ is injective for almost all $p$. Hence, $\ulim_p H_{\m_p}^d(R_p)\to \ulim_p H_{\m_p}^d(B_p)$ is injective. Therefore, $\lbrack\frac{z_p}{x_p^t}\rbrack=0$ in $H_{\m_p}^d(R_p)$ for almost all $p$. Since $R_p$ is Cohen-Macaulay for almost all $p$, we have $z_p\in (x_{1p}^t,\dots,x_{dp}^t)$ for almost all $p$. Hence, $z\in (x_1^t,\dots,x_d^t)$ by \L o\'{s}'s theorem. Therefore, $H_\m^d(R)\to H_\m^d(B)$ is injective.
	Conversely, suppose that $R$ is $\text{BCM}_B$-rational. Let $I=(x_1,\dots,x_d)$ be an ideal generated by the system of parameters. Let $z\in I^{*\gen}$. Since $I_p^*\subseteq I_pB_p\cap R_p$ by \cite[Theorem 5.1]{plus closure} for almost all $p$, we have $\lbrack\frac{z_p}{x_p}\rbrack=0$ in $H_{\m_p}^d(B_p)$ for almost all $p$. Since $H_\m^d(B)\to \ulim_p H_{\m_p}^d(B_p)$ and $H_\m^d(R)\to H_\m^d(B)$ are injective, we have $\lbrack\frac{z}{x}\rbrack=0$ in $H_\m^d(R)$. Since $R$ is Cohen-Macaulay, $z\in I$. Therefore, $R$ is generically $F$-rational. By Proposition \ref{generically F-rational rational} (see \cite[Theorem 6.2]{Puresubrings}), $R$ has rational singularities.
\end{proof}
\section{Approximations of multiplier ideals}
In this section, we will explain the relation between approximations and reductions modulo $p\gg 0$.
 Note that an isomorphism $\ulim_p \bar{\F_p}\cong \C$ is fixed.
 \begin{defn}
 	Let $R$ be a finitely generated $\C$-algebra. A pair $(A,R_A)$ is called a {\it model of $R$} if the following two conditions hold:
 	\begin{enumerate}[(i)]
 		\item $A \subseteq \C$ is a finitely generated $\Z$-subalgebra.
 		\item $R_A$ is a finitely generated $A$-algebra such that $R_A\otimes_A\C\cong R$.
 	\end{enumerate}
 \end{defn}
\begin{prop}[{\cite[Lemma 4.10]{canonical BCM}}]\label{model of the field}
	Let $A$ be a finitely generated $\Z$-subalgebra of $\C$. There exists a family $(\gamma_p)_p$ which satisfies the following two conditions:
	\begin{enumerate}[(i)]
		\item $\gamma_p:A\to \bar{\F_p}$ is a ring homomorphism for almost all $p$.
		\item For any $x\in A$, $x=\ulim_p\gamma_p(x)$.
	\end{enumerate}
\end{prop}
\begin{prop}[{cf. \cite[Corollary 4.10]{canonical BCM}}]\label{model of affine ring}
	Let $R$ be a finitely generated $\C$-algebra and let ${\bf a}=a_1,\dots, a_l$ be finitely many elements of $R$. Let $R_p$ be an approximation of $R$. Then there exists a model $(A,R_A)$ which satisfies the following conditions:
	\begin{enumerate}[(i)]
		\item There exists a family $(\gamma_p)$ as in Proposition \ref{model of the field}.
		\item ${\bf a}\subseteq R_A$.
		\item $R_A\otimes_A \bar{\F_p}\cong R_p$ for almost all $p$.
		\item For any $x\in R_A$, the ultraproduct of the image of $x$ under $\id_{R_A}\otimes_A \gamma_p$ is $x$.
	\end{enumerate}
\end{prop}
\begin{proof}
	Let $X=X_1,\dots,X_n$ and $R\cong \C[X]/I$ for some ideal $I\subseteq \C[X]$. Take any model $(A,R_A)$ which contains ${\bf a}$. Enlarging this model, we may assume that there exits an ideal $I_A\subseteq A[X]$ such that $R_A\cong A[X]/I_A$ and $I_A\otimes_A \C=I$ in $\C[X]$. Take $(\gamma_p)$ as in Proposition \ref{model of the field}. Let $I=(f_1,\dots, f_m)$. For $f=\sum_\nu c_\nu X^\nu \in A[X]\subseteq \C[X]$, by the definition of approximations, $f_p:=\sum_\nu \gamma_p(c_\nu)X^\nu \in \bar{\F_p}[X]$ is an approximation of $f$. Hence, by the definition of approximations of finitely generated $\C$-algebras, $R_A\otimes_A\bar{\F_p}\cong \bar{\F_p}[X]/(f_{1p},\dots,f_{mp})\bar{\F_p}[X]$ is an approximation of $R$. Since two approximations are isomorphic for almost all $p$, $R_A\otimes_A\bar{\F_p}\cong R_p$ for almost all $p$. The condition (iv) is clear by the above argument.
\end{proof}
\begin{rem}\label{model of local rings}
	Let $\p =(x_1,\dots,x_n) \subseteq R$ be a prime ideal. Enlarging the model $(A,R_A)$, we may assume that $x_1,\dots,x_n\in R_A$.
	 Let $\mu_p$ be the kernel of $\gamma_p:A\to \bar{\F_p}$. Then this is a maximal ideal of $A$ and $A/\mu_p$ is a finite field. $\p_{\mu_p}=(x_1,\dots,x_n) R_A/\mu_p R_A $ is prime for almost all $p$ since this is a reduction to $p\gg0$. On the other hand, $\p_p:=(x_1,\dots,x_n)R_A\otimes_A\bar{\F_p}\subseteq R_p$ is an approximation of $\p$. Hence, $\p_p$ is prime for almost all $p$. Here, $(R_p)_{\p_p}$ is an approximation of $R_\p$. Thus we have a flat local homomorphism $(R_A/\mu_p R_A)_{\p_{\mu_p}}\to R_p$ with $\p_{\mu_p}R_p=\p_p$. Moreover, if $\p$ is maximal, then $\p_{\mu_p},\p_p$ are maximal for almost all $p$. Then, the map $R_A/\p_{\mu_p}\to R_p/\p_p\cong \bar{\F_p}$ is a separable field extension since $R_A/\p_{\mu_p}$ is a finite field.
\end{rem}
The next result is a generalization of \cite[Theorem 4.6]{TY} from ideal pairs to triples.
\begin{prop}\label{multiplier approximation}
	Let $R$ be a normal local domain essentially of finite type over $\C$, $\Delta\ge 0$ an effective $\Q$-Weil divisor such that $K_R+\Delta$ is $\Q$-Cartier, $\ba$ a nonzero ideal and $t>0$ a real number. Suppose that $R_p$, $\Delta_p$, $\ba_p$ are approximations. Then $\tau(R_p,\Delta_p,\ba_p^t)$ is an approximation of $\mathcal{J}(\Spec R,\Delta,\ba^t)$.
\end{prop}
\begin{proof}
	Let $R=S_\p$, where $S$ is a normal domain of finite type over $\C$ and $\p$ is a prime ideal. Let $\m$ be a maximal ideal contains $\p$. Then there exists a model $(A,S_A)$ of $S$ such that the properties in Proposition \ref{model of affine ring} hold and $S_A$ containing a system of generators of $\J(\Spec R,\Delta,\ba^t)$ and $\Delta_A$, $\ba_A$ can be defined properly. Let $\mu_p$ be maximal ideals of $S_A$ as in Remark \ref{model of local rings} and let $\m_{\mu_p}, \p_{\mu_p}$ be reductions to $p\gg 0$. Since, for almost all $p$, $(S_A/\mu_p)_{\m_{\mu_p}}\to (S_\m)_p$ is a flat local homomorphism such that $S_A/\m_{\mu_p}\to (S/\m)_p\cong \bar{\F_p}$ is a separable field extension, we have
	\[
	\tau((S_A/\mu_p)_{\m_{\mu_p}}, \Delta_{(S_A/\mu_p)_{\m_{\mu_p}}},\ba^t_{(S_A/\mu_p)_{\m_{\mu_p}}})(S_\m)_p=\tau((S_\m)_p,\Delta_{\m_p},\ba_{\m_p}^t),
	\]
	by a generalization of \cite[Lemma 1.5]{Srini-Tak}.
	Since the localization commutes with test ideals (\cite[Proposition 3.1]{HT04}), we have
	\[
	\tau((S_A/\mu_p)_{\p_{\mu_p}}, \Delta_{(S_A/\mu_p)_{\p_{\mu_p}}},\ba^t_{(S_A/\mu_p)_{\p_{\mu_p}}})R_p=\tau(R_p,\Delta_p,\ba_p^t)
	\]
	for almost all $p$.
	Since the reduction of multiplier ideals modulo $p\gg 0$ is the test ideal (\cite[Theorem 3.2]{interpret. multiplier}), $\tau((S_A/\mu_p)_{\p_{\mu_p}}, \Delta_{(S_A/\mu_p)_{\p_{\mu_p}}},\ba^t_{(S_A/\mu_p)_{\p_{\mu_p}}})$ is a reduction of 
	\[
	\J(\Spec R,\Delta,\ba^t)
	\]
	 to characteristic $p\gg 0$. Hence, $\tau(R_p,\Delta_p,\ba^t_p)$ is an approximation of $\J(\Spec R,\Delta,\ba^t)$.
\end{proof}

\section{BCM test ideal with respect to a big Cohen-Macaulay algebra constructed via ultraproducts}
Throughout this section, we assume that $(R,\m)$ is a normal local domain essentially of finite type over $\C$. Fix a canonical divisor $K_R$ such that $R\subseteq \omega_R:=R(K_R)\subseteq \operatorname{Frac}(R)$. Let $\Delta \ge 0$ be an effective $\Q$-Weil divisor such that $K_R+\Delta$ is $\Q$-Cartier. Suppose that $\Div f=n(K_R+\Delta)$ for $f\in R^{\circ}$, $n\in \N$. Let $B_p$ be a big Cohen-Macaulay $R_p^+$-algebra for almost all $p$ and $B:=\ulim_p B_p$. We use $\widehat{R}$ to denote the completion of $R$ with respect to $\m$ and $\widehat{\Delta}$ to denote the flat pullback of $\Delta$ by $\Spec \widehat{R}\to\Spec R$.

\begin{prop}\label{multiplier subset BCM-test}
	In the setting as above, we have
	\[
	\J(\widehat{R},\widehat{\Delta})\subseteq \tau_{\widehat{B}}(\widehat{R},\widehat{\Delta}).
	\]
\end{prop}
\begin{proof}
	Consider the following commutative diagram:
	\[
	\xymatrix{	
	0_{H_\m^d(\omega_R)}^{B,K_R+\Delta} \ar[r] \ar[d] & \ulim_p 0_{H_{\m_p}^d(\omega_{R_p})}^{B_p,K_{R_p}+\Delta_p} \ar[d] \\
	H_\m^d(\omega_R) \ar[r] \ar[d]_{\psi} & \ulim_p H_{\m_p}^d(\omega_{R_p}) \ar[d] \\
	H_\m^d(B) \ar[r] & \ulim_p H_{\m_p}^d(B_p) \\
	}.
	\]
	By Proposition \ref{BCM test tight closure}, we have
	\[
	0_{H_{\m_p}^d(\omega_{R_p})}^{B_p,K_{R_p}+\Delta_p}=0_{H_{\m_p}^d(\omega_{R_p})}^{*\Delta_p}
	\]
	for almost all $p$. Let $x_1,\dots,x_d$ be a system of parameters and $x=x_1\cdots x_d$ the product of them.
	Take $a\in\J(R,\Delta)=\ulim_p \tau(R_p,\Delta_p)\cap R$ and $\lbrack \frac{z}{x^t}\rbrack \in 0_{H_{\m}^d(\omega_{R})}^{B,K_R+\Delta}$. Let $J$ be a divisorial ideal which is isomorphic to $\omega_R$ and $g\in R^\circ$ an element such that $\omega_R \xrightarrow{\cdot g} J$ is an isomorphism. As in Proof of \cite[Theorem 2.8]{interpret. multiplier}, we have $g_p z_p x_p^t\in ((x_{1p}^{2t},\dots,x_{dp}^{2t})J_p)^{*\Delta_p}$ for almost all $p$. Hence, $a_p g_p z_p x_p^t\in (x_{1p}^{2t},\dots,x_{dp}^{2t})J_p$ for almost all $p$. Therefore, $agzx^t\in (x_1^{2t},\dots,x_d^{2t})J$ and $\lbrack \frac{az}{x^t} \rbrack =0$ in $H_\m^d(\omega_R)$. Hence, we have $a\in \Ann_R 0_{H_{\m}^d(\omega_{R})}^{B,K_R+\Delta}$. In conclusion, we have $\J (R,\Delta)\widehat{R}\subseteq \tau_{\widehat{B}}(\widehat{R},\widehat{\Delta})$.
\end{proof}
\begin{lem}[{\cite[Theorem 2.13]{interpret. multiplier}}]\label{lemma of Main thm}
	Let $(R,\m)$ be an $F$-finite normal local domain of characteristic $p>0$ and $\Delta\ge 0$ be an effective $\Q$-Weil divisor on $X:=\Spec R$ such that $K_X+\Delta$ is $\Q$-Cartier. Let $f:Y\to X$ be a proper birational morphism with $X$ normal. Suppose that $Z:=f^{-1}(\m)$ and $\delta:H_\m^d(R(K_X))\to H_Z^d(Y,\sO_Y(\lfloor f^{*}(K_X+\Delta)\rfloor)$ is the Matlis dual of the natural inclusion map $H^0(Y,\sO_Y(\lceil K_Y-f^*(K_X+\Delta) \rceil))\hookrightarrow R$. Then $\Ker \delta\subseteq 0_E^{*\Delta}$, where $E$ is the injective hull of the residue field $R/\m$ of $R$.
\end{lem}
\begin{proof}
	By \cite[Theorem 2.13]{interpret. multiplier}, we have $\tau(R,\Delta)\subseteq H^0(Y,\sO_Y(\lceil K_Y-f^*(K_X+\Delta) \rceil))$. Hence,
	\begin{eqnarray*}
		\Ker \delta&=&\Ann_E H^0(Y,\sO_Y(\lceil K_Y-f^*(K_X+\Delta) \rceil))\\
		&\subseteq& \Ann_E \tau(R,\Delta) \\
		&=&\Ann_E \tau(R,\Delta)\widehat{R} \\
		&=&\Ann_E \tau(\widehat{R},\widehat{\Delta}) \\
		&=&\Ann_E \Ann_{\widehat{R}} 0_E^{*\Delta} \\
		&=& 0_E^{*\Delta}.
	\end{eqnarray*}
\end{proof}
\begin{rem}
	Moreover, we have $\Ker \delta=0_E^{*\Delta}$ if $f$ is a reduction of a log resolution in characteristic zero modulo $p\gg 0$ by \cite[Theorem 3.2]{interpret. multiplier}.
\end{rem}
\begin{thm}\label{Main Theorem}
	Let $R$ be a normal local domain essentially of finite type over $\C$. Fix an effective canonical divisor $K_R\ge 0$ on $\Spec R$. Let $\Delta \ge 0$ be an effective $\Q$-Weil divisor on $\Spec R$ such that $K_R+\Delta$ is $\Q$-Cartier and $B_p$ a big Cohen-Macaulay $R_p^+$-algebra for almost all $p$. Suppose that $n(K_R+\Delta)=\Div (f)$ for $f\in R^\circ, n\in \N$. Then we have
	\[
		\tau_{\widehat{B}}(\widehat{R},\widehat{\Delta})=\J(\widehat{R},\widehat{\Delta}).
	\]
\end{thm}
\begin{proof}
	Thanks to Proposition \ref{multiplier subset BCM-test}, it suffices to prove $\tau_{\widehat{B}}(\widehat{R},\widehat{\Delta})\subseteq \J(\widehat{R},\widehat{\Delta})$. Let $\mu:Y\to X:=\Spec R$ be a log resolution of $(X,\Delta)$ and let $Z:=\mu^{-1}(\m)$. Considering approximations, we have a corresponding morphisms $\mu_p:Y_p\to X_p:=\Spec R_p$, $Z_p=\mu_p^{-1}(\m_p)$ for almost all $p$. Then we have a commutative diagram
	\[
	\xymatrix{
		&H_\m^d(\omega_R) \ar[d]^{\gamma} \ar[rd]^{\delta}& \\
		H^{d-1}(Y,\sL)\ar[d] \ar[r] & H^{d-1}(Y\setminus Z,\sL|_{Y\setminus Z}) \ar [r] \ar[d]^{u^{d-1}} & H_Z(\sL) \\
		H_\infty^{d-1}(Y,\sL) \ar[r]^-{\rho_\infty^{d-1}} & H_\infty^{d-1}(Y\setminus Z,\sL|_{Y\setminus Z})
	},
	\]
	where $\sL:=\sO_Y(\lfloor\mu^*(K_X+\Delta) \rfloor)$ and the middle row is exact.
	Similarly, we have the following commutative diagram for almost all $p$:
	\[
	\xymatrix{
	& H_{\m_p}^d(\omega_{R_p}) \ar[d]^{\gamma_p} \ar[rd]^{\delta_p} & \\
	H^{d-1}(Y_p,\sL_p) \ar[r]^-{\rho_p^{d-1}} & H^{d-1}(Y_p\setminus Z_p, \sL_p|_{Y_p\setminus Z_p}) \ar[r] & H^d_{Z_p}(\sL_p)
	},
	\]
	where the middle row is exact.
	Assume that $\eta \in \Ker \delta$. Then $u^{d-1}(\gamma(\eta))\in \operatorname{Im} \rho_\infty^{d-1}$. Therefore, $\gamma_p(\eta_p) \in \operatorname{Im}\rho_p^{d-1}$ for almost all $p$. Hence, $\eta_p\in \Ker \delta_p$ for almost all $p$. By lemma \ref{lemma of Main thm}, $\eta_p\in 0_{H_{\m_p}^d(\omega_{R_p})}^{*\Delta_p}$ for almost all $p$. Hence, by Propositon \ref{BCM test tight closure}, we have $\eta_p\in 0_{H_{m_p}^d(\omega_{R_p})}^{B_p,K_{R_p}+\Delta_p}$ for almost all $p$. We have a commutative diagram
	\[
	\xymatrix{
		H_\m^d(\omega_R) \ar[r] \ar[d]^{\psi} & \ulim_p H_{\m_p}^d(\omega_{R_p}) \ar[d]^{\psi_\infty:=\ulim_p\psi_p} \\
		H_\m^d(B) \ar[r] & \ulim_p H_{\m_p}^d(B_p)
	},
	\]
	where $\psi$, $\psi_p$ is the morphisms as in Definition \ref{def of BCM test ideal}.
	Since $\psi_\infty(\ulim_p \eta_p)=0$ and $H_\m^d(B) \to \ulim_p H_{\m_p}^d(B_p)$ is injective by Lemma \ref{injectivity local cohomology}, we have $\psi(\eta)=0$ in $H_\m^d(B)$. Hence, $\eta\in 0_{H_\m^d(\omega_R)}^{B,K_R+\Delta}$. Therefore, we have
	\begin{eqnarray*}
	\tau_{\widehat{B}}(\widehat{R},\widehat{\Delta})&\subseteq& \Ann_{\widehat{R}}(\Ker \delta)\\
	&=& \Ann_{\widehat{R}}\Ann_{H_\m^d(\omega_R)}\J(R,\Delta)\\
	&=& \J(\widehat{R},\widehat{\Delta}).
	\end{eqnarray*}
\end{proof}
\begin{rem}
	We can generalize the notion of ultra-test ideals in \cite[Definition 5.5]{TY} to the pair $(R,\Delta)$.
	Using Lemma \ref{lemma of Main thm} instead of \cite[Theorem 6.9]{HY}, we can show that generalized ultra-test ideals are equal to multiplier ideals.
\end{rem}
\section{Generalized module closures and applications}
  We introduce the notion of generalized module closures inspired by \cite{PRG21}. Using the generalized module closures, we will generalize \cite[Corollary 5.30]{TY}. We also use \cite[Subsection 6.1]{Ma-Schwede} as reference in the following arguments.
  \begin{setting}\label{Setting 7}
  Suppose that $R$ is a normal local domain essentially of finite type over $\C$ of dimension $d$, $K_R\ge0$ is a fixed effective canonical divisor and $\Delta\ge 0$ is an effective $\Q$-Weil divisor such that $K_R+\Delta$ is $\Q$-Cartier. Moreover, we assume that $B_p$ is a big Cohen-Macaulay $R_p^+$-algebra for almost all $p$, $B:=\ulim_p B_p$ and $r(K_R+\Delta)=\Div f$ for $f\in R$, $r\in\N$. Let $R'\subseteq R^+$ be an integrally closed finite extension of $R$ such that $f^{1/r}\in R'$ and $\pi^*\Delta$ is Weil divisor, where $\pi:\Spec R'\to \Spec R$.
  \end{setting}
 \begin{defn}\label{generalized closure}
 	Assume Setting \ref{Setting 7} and let $g\in R^\circ$ and $t>0$ be a positive rational number. We use $\widehat{B_\Delta}$ to denote
 	\[
 	B\otimes_{R'}R'(\pi^*\Delta)\otimes_R \widehat{R}.
 	\]
 	For any $\widehat{R}$-modules $N\subseteq M$, we define $N_M^{\cl_{\widehat{B_\Delta},g^t}}$ as follows: $x\in N_M^{\cl_{\widehat{B_\Delta},g^t}}$ if and only if $g^t\otimes x\in \operatorname{Im}(\widehat{B_\Delta}\otimes_{\widehat{R}}N\to\widehat{B_\Delta}\otimes_{\widehat{R}}M)$.
 	We use $\tau_{\cl_{\widehat{B_\Delta},g^t}}(\widehat{R})$ to denote
 	\[
 	\bigcap_{N\subseteq M}(N:_{\widehat{R}}N_M^{\cl_{\widehat{B_\Delta}},g^t}),
 	\]
 	where $M$ runs through all $\widehat{R}$-modules and $N$ runs through all $\widehat{R}$-submodules of $M$.
 \end{defn}
\begin{prop}\label{Prop test ideal of module closure}
	In Setting \ref{Setting 7}, if $g\in R^\circ$ and $t>0$ is a positive rational number, then we have
	\[
	\tau_{\cl_{\widehat{B_\Delta},g^t}}(\widehat{R})=\bigcap_{M}\Ann_{\widehat{R}}0_M^{\cl_{\widehat{B_\Delta},g^t}}=\Ann_{\widehat{R}}0_E^{\cl_{\widehat{B_\Delta},g^t}},
	\]
	where $M$ runs through all $\widehat{R}$-modules and $E$ is the injective hull of the residue field of $R$.
\end{prop}
\begin{proof}
	We can prove this by arguments similar to \cite[Lemma 3.3, Proposition 3.9]{PRG21}.
\end{proof}
\begin{prop}\label{BCM closure and module closure}
	In Setting \ref{Setting 7}, if $g\in R^\circ$ and $t>0$ is a positive rational number, then we have
	\[
	0_E^{B,K_R+\Delta+t\Div g}=0_E^{\cl_{\widehat{B_\Delta},g^t}}.
	\]
\end{prop}
\begin{proof}
	Since the reflexive hull $(R'(\pi^*\Delta)\otimes_R \omega_R)^{**}$ is equal to $R'(\Div (f^{\frac{1}{r}}))$, we have $H_\m^d(R'(\pi^*\Delta)\otimes_R \omega_R)\cong H_\m^d(R'(\Div (f^{\frac{1}{r}})))$. Hence, we have
	\begin{eqnarray*}
		\widehat{B_\Delta}\otimes_{\widehat{R}} E &\cong& B\otimes_{R'}H_\m^d(R'(\pi^*\Delta)\otimes_R\omega_R) \\
		&\cong& B\otimes_{R'}H_\m^d(R'(\Div (f^{\frac{1}{r}}))).
	\end{eqnarray*}
	Then there exists a commutative diagram
	\[
	\xymatrix{
		E\cong H_\m^d(\omega_R) \ar[r] \ar[dddd] & \widehat{B_\Delta}\otimes_{\widehat{R}}E \ar[d]^{g^t\otimes 1} \\
		& \widehat{B_\Delta}\otimes_{\widehat{R}}E \ar[d]^{\cong} \\
		& B\otimes_{R'}H_\m^d(R'(\Div (f^{\frac{1}{r}}))) \ar[d]^{\id \otimes (\cdot f^{1/r})} \\
		& B\otimes_{R'} H_\m^d(R') \ar[d]^{\cong}\\
		H_\m^d(B\otimes_R  \omega_R) \ar[r]_{\psi}& H_\m^d(B)
	},
	\]
	where $\psi$ is the second map of
	\[
	\cdot f^{\frac{1}{r}}g^t : H_\m^d(B) \to H_\m^d(B\otimes_R\omega_R) \to H_\m^d(B).
	\]
	The result follows by the above commutative diagram.
\end{proof}
\begin{defn}\label{closure with ideal}
	Let $R\hookrightarrow S$ be an injective local homomorphism of normal local domains essentially of finite type over $\C$. Fix $K_R, K_S\ge0$ effective canonical divisors on $\Spec R$ and on $\Spec S$, respectively. Let $\Delta_R,\Delta_S\ge 0$ be effective $\Q$-Weil divisors on $\Spec R$ and on $\Spec S$, respectively, such that $K_R+\Delta_R$, $K_S+\Delta_S$ are $\Q$-Cartier. Let $\ba\subseteq R$ be a nonzero ideal and $t>0$ be a positive rational number. Suppose that $\widehat{B_{\Delta_R}}$ and $\widehat{B_{\Delta_S}}$ are defined as in Definition \ref{generalized closure}. Then, for an $\widehat{R}$-module $M$ and an $\widehat{S}$-module $N$, we define $0_{M}^{\cl_{\widehat{B_{\Delta_R}},\ba^t}}$, $0_{N}^{\cl_{\widehat{B_{\Delta_S}}\ba^t}}$ by
	\begin{eqnarray*}
		0_{M}^{\cl_{\widehat{B_{\Delta_R}},\ba^t}}:=\bigcap_{n\in\N}\bigcap_{g\in\ba^{\lceil nt \rceil}}0_{M}^{\cl_{\widehat{B_{\Delta_R}},g^{\frac{1}{n}}}},\\
		0_{N}^{\cl_{\widehat{B_{\Delta_S}}\ba^t}}:=\bigcap_{n\in\N}\bigcap_{g\in\ba^{\lceil nt \rceil}}0_{N}^{\cl_{\widehat{B_{\Delta_S}}g^{\frac{1}{n}}}}.
	\end{eqnarray*}
	We use $\tau_{\cl_{\widehat{B_{\Delta_R}},\ba^t}}(\widehat{R})$, $\tau_{\cl_{\widehat{B_{\Delta_S}},\ba^t}}(\widehat{S})$ to denote
	\[
	\bigcap_M \Ann_{\widehat{R}}0_M^{\cl_{\widehat{B_{\Delta_R}},\ba^t}}, \\
	\bigcap_N \Ann_{\widehat{S}}0_N^{\cl_{\widehat{B_{\Delta_S}},\ba^t}},
	\]
	where $M$ runs through all $\widehat{R}$-modules and $N$ runs through all $\widehat{S}$-modules.
\end{defn}
\begin{prop}
	In the setting of Definition \ref{closure with ideal}, we have
	\begin{eqnarray*}
		\Ann_{\widehat{R}}0_{E_R}^{\cl_{\widehat{B_{\Delta_R}},\ba^t}}=\bigcap_M \Ann_{\widehat{R}}0_M^{\cl_{\widehat{B_{\Delta_R}},\ba^t}}, \\
		\Ann_{\widehat{S}}0_{E_S}^{\cl_{\widehat{B_{\Delta_S}},\ba^t}}=\bigcap_N \Ann_{\widehat{S}}0_N^{\cl_{\widehat{B_{\Delta_S}},\ba^t}},
	\end{eqnarray*}
	where $M,N$ run through all $\widehat{R}$-modules and all $\widehat{S}$-modules, respectively, and $E_R$, $E_S$ are the injective hulls of the residue fields of $R$ and of $S$, respectively.
\end{prop}
\begin{proof}
	We can show this by arguments similar to Proposition \ref{Prop test ideal of module closure}.
\end{proof}
\begin{prop}
	In the setting of Definition \ref{closure with ideal}, we have
	\[
	\tau_{\cl_{\widehat{B_{\Delta_R}},\ba^t}}(\widehat{R})=\J(\widehat{R},\widehat{\Delta},(\ba\widehat{R})^t).
	\]
\end{prop}
\begin{proof}
	Let $E$ be the injective hull of the residue field of $R$. Then
	\begin{eqnarray*}
		0_E^{\cl_{\widehat{B_{\Delta_R}},\ba^t}} &=& \bigcap_{n\in \N}\bigcap_{g\in\ba^{\lceil nt \rceil}} 0_E^{\cl_{\widehat{B_{\Delta_R}}, g^{\frac{1}{n}}}} \\
		&=& \bigcap_{n\in\N}\bigcap_{g\in\ba^{\lceil nt \rceil}}\Ann_E \J(\widehat{R},\widehat{\Delta},g^{\frac{1}{n}}) \\
		&=& \Ann_E \sum_{n\in \N}\sum_{g\in\ba^{\lceil nt \rceil}}\J (\widehat{R},\widehat{\Delta},g^{\frac{1}{n}}) \\
		&=& \Ann_E \J(\widehat{R},\widehat{\Delta},(\ba \widehat{R})^t),
	\end{eqnarray*}
	where the second equality follows from Theorem \ref{Main Theorem}.
	Hence, we have
	\[
	\Ann_{\widehat{R}}0_E^{\cl_{\widehat{B_{\Delta_R}},\ba^t}}=\J(\widehat{R},\widehat{\Delta},(\ba \widehat{R})^t).
	\]
\end{proof}
The next lemma is a generalization of \cite[Theorem 3.2]{Formulas}.
\begin{lem}
	Let $R$ be a normal local domain essentially of finite type over $\C$, $\Delta\ge 0$ an effective $\Q$-Weil divisor such that $K_R+\Delta$ is $\Q$-Cartier. Let $\ba_1,\dots,\ba_n\subseteq R$ be nonzero ideals and $t>0$ a positive rational number.
	Then we have
	\[
	\J(R,\Delta,(\ba_1+\dots+\ba_n)^t)=\sum_{\lambda_1+\dots+\lambda_n=t}\J(R,\Delta,\ba_1^{\lambda_1}\cdots \ba_n^{\lambda_n}).
	\]
\end{lem}
\begin{lem}
	In the setting of Definition \ref{closure with ideal}, we have
	\[
	\sum_{n\in\N}\sum_{g\in\ba^{\lceil nt \rceil}}\J(S,\Delta_S,g^{\frac{1}{n}})=\J(S,\Delta_S,(\ba S)^t).
	\]
\end{lem}
\begin{proof}
	$\sum_{n\in\N}\sum_{g\in\ba^{\lceil nt \rceil}}\J(S,\Delta_S,g^{1/n})\subseteq \J(S,\Delta_S,(\ba S)^t)$ is clear.
	If $t=q/p$, $p,q>0$ and $\ba=(g_1,\dots,g_l)$, then
	\begin{eqnarray*}
		\sum_{n\in\N}\sum_{g\in\ba^{\lceil nt \rceil}}\J(S,\Delta_S,g^{\frac{1}{n}}) & \supseteq & \sum_{n\in \N}\sum_{i_1+\dots+i_l=nq}\J(S,\Delta_S,(g_1^{i_1}\cdots g_l^{i_l})^{\frac{1}{np}}) \\
		&=&\J(S,\Delta_S,(\ba S)^t),
	\end{eqnarray*}
	by the above lemma.
\end{proof}
\begin{thm}\label{a behavior under pure ring extensions}
	Let $R\hookrightarrow S$ be a pure local homomorphism of normal local domains essentially of finite type over $\C$. Fix effective canonical divisors $K_R$ and $K_S$ on $\Spec R$ and on $\Spec S$, respectively. Let $\Delta_R, \Delta_S \ge0$ be effective $\Q$-Weil divisors on $\Spec R$, $\Spec S$ such that $K_R+\Delta_R$, $K_S+\Delta_S$ are $\Q$-Cartier. Take normal domains $R',S'$ and morphisms $\pi_R,\pi_S$ as in Setting \ref{Setting 7}. Moreover, let $\ba\subseteq R$ be a nonzero ideal and $t>0$ a positive rational number. If $R'(\pi_R^*\Delta_R)\subseteq S'(\pi_S^*\Delta_S)$, then we have
	\[
	\J(S,\Delta_S,(\ba S)^t)\cap R\subseteq \J(R,\Delta_R,\ba^t).
	\]
\end{thm}
\begin{proof}
	Since $R\hookrightarrow S$ is pure, $\widehat{R} \hookrightarrow \widehat{S}$ is pure, see \cite[Corollary 3.2.1]{Puresubrings of comm}. Since $R\to\widehat{R}$, $S\to\widehat{S}$ are pure, it is enough to show
	\[
	\J(\widehat{S},\widehat{\Delta_S},(\ba \widehat{S})^t)\cap \widehat{R}\subseteq \J(\widehat{R},\widehat{\Delta_R},(\ba \widehat{R})^t).
	\]
	Let $\B(R)$, $\B(S)$ be the canonical big Cohen-Macaulay algebras. Let $\widehat{\B_{\Delta_R}}:=\widehat{\B(R)_{\Delta_R}}$ and $\widehat{\B_{\Delta_S}}:=\widehat{\B(S)_{\Delta_S}}$. Take an $\widehat{R}$-module $M$. Then we have a commutative diagram
	\[
	\xymatrix{
		\widehat{R} \ar@{^{(}->}[r]^{\text{pure}} \ar[d] & \widehat{S} \ar[d] \\
		\widehat{\B_{\Delta_R}} \ar[r] & \widehat{\B_{\Delta_S}}
	}.
	\]
	Tensoring the commutative diagram with $M$, we have
	\[
	\xymatrix{
		M \ar@{^{(}->}[r] \ar[d] & \widehat{S}\otimes_{\widehat{R}}M \ar[d] \\
		\widehat{\B_{\Delta_R}}\otimes_{\widehat{R}}M \ar[r] & \widehat{\B_{\Delta_S}}\otimes_{\widehat{R}}M
	}.
	\]
	Hence, we have
	\[
	0_{M}^{\cl_{\widehat{\B_{\Delta_R}},\ba^t}}\subseteq 0_{\widehat{S}\otimes_{\widehat{R}}M}^{\cl_{\widehat{\B_{\Delta_S}},\ba^t}}.
	\]
	Then we have
	\begin{eqnarray*}
		\J(\widehat{R},\widehat{\Delta_R},\ba^t)&=&\bigcap_{M} \Ann_{\widehat{R}}0_M^{\cl_{\widehat{\B_{\Delta_R}},\ba^t}} \\
		&\supseteq & \bigcap_{M}\Ann_{\widehat{R}}0_{M\otimes S}^{\cl_{\widehat{\B_{\Delta_S}},\ba^t}} \\
		&\supseteq& \bigcap_N \Ann_{\widehat{R}}0_N^{\cl_{\widehat{\B_{\Delta_S}},\ba^t}}\\
		&=& \bigcap_N (\Ann_{\widehat{S}}0_N^{\cl_{\widehat{\B_{\Delta_S}},\ba^t}}\cap \widehat{R}) \\
		&=& (\Ann_{\widehat{S}}0_{E_S}^{\cl_{\widehat{\B_{\Delta_S}},\ba^t}})\cap \widehat{R} \\
		&=& (\Ann_{\widehat{S}}\bigcap_{n\in\N}\bigcap_{g\in\ba^{\lceil nt \rceil}} 0_{E_S}^{\cl_{\widehat{\B_{\Delta_S}},g^{\frac{1}{n}}}})\cap \widehat{R} \\
		&=& (\Ann_{\widehat{S}}\Ann_{E_S}\sum_{n\in\N}\sum_{g\in \ba^{\lceil nt \rceil}}\J(\widehat{S},\widehat{\Delta_S},g^{\frac{1}{n}}))\cap \widehat{R} \\
		&=& \J(\widehat{S},\widehat{\Delta_S},(\ba \widehat{S})^t))\cap \widehat{R},
	\end{eqnarray*}
	where $M$ runs through all $\widehat{R}$-modules, $N$ runs through all $\widehat{S}$-modules and $E_S$ is the injective hull of the residue field of $S$.
\end{proof}
As a corollary, we have a generalization of \cite[Corollary 5.30]{TY} to the case that $\ba$ is not necessarily a principal ideal.
\begin{cor}\label{Main Thm 2}
		Let $R\hookrightarrow S$ be a pure local homomorphism of normal local domains essentially of finite type over $\C$. Suppose that $R$ is $\Q$-Gorenstein. Fix effective canonical divisors $K_R$ and $K_S$ on $\Spec R$ and on $\Spec S$, respectively. Let $\Delta_S$ be an effective $\Q$-Weil divisor on $\Spec S$ such that $K_S+\Delta_S$ is $\Q$-Cartier. Let $\ba\subseteq R$ be a nonzero ideal and $t>0$ a positive rational number. Then we have
	\[
	\J(S,\Delta_S,(\ba S)^t)\cap R\subseteq \J(R,\ba^t).
	\]
\end{cor}
\begin{proof}
	Let $R'$ be the integral closure of $R[f^{1/r}]$ in $R^+$. Then the result follows from Theorem \ref{a behavior under pure ring extensions}.
\end{proof}

\section{$\B$-regularity}
 As another application of the main theorem, we will give a partial answer to \cite[Remark 3.10]{log-terminal}. For this, we will review the definition of $\B$-regularity.
\begin{defn}[{\cite[Definition 4.3]{canonical BCM}}]
	Let $R$ be a normal $\Q$-Gorenstein local domain essentially of finite type over $\C$.
	\begin{enumerate}
		\item $R$ is said to be {\it weakly $\B$-regular} if $R\to\B(R)$ is cyclically pure.
		\item $R$ is said to be {\it $\B$-regular} if every localization of $R$ at a prime ideal is weakly $\B$-regular.
	\end{enumerate}
\end{defn}
 \begin{thm}\label{partial answer to Schoutens}
 	Let $R$ be a normal $\Q$-Gorenstein local domain. Then the following are equivalent:
 	\begin{enumerate}
 		\item $R$ has log-terminal singularities.
 		\item $R$ is ultra-$F$-regular.
 		\item $R$ is weakly generically $F$-regular.
 		\item $R$ is generically $F$-regular.
 		\item $R$ is weakly $\B$-regular.
 		\item $R$ is $\B$-regular.
 		\item $\widehat{R}$ is $\text{BCM}_{\widehat{\B(R)}}$-regular.
 	\end{enumerate}
 \end{thm}
\begin{proof}
	The equivalence of (1) and (2) follows from Proposition \ref{ultra-F-regular log-terminal} and the equivalence of (1) and (7) follows from Theorem \ref{Main Theorem}.
	Since if $R$ has log-terminal singularities, then every localization of $R$ at a prime ideal is log-terminal, it is enough to show the equivalence of (1), (3) and (5). (1) is equivalent to (3) by \cite[Theorem 5.24, Proof of Theorem 5.25]{TY}. Lastly, we will show the equivalence of (5) and (7). Let $E$ be the injective hull of the residue field of $R$. By Proposition \ref{BCM closure and module closure}, we have $0_E^{\cl_{\B(R)\otimes_R\widehat{R}}}=0_E^{\B(R),K_R}$. Hence, $E\to \B(R)\otimes_R E$ is injective if and only if $\widehat{R}$ is $\text{BCM}_{\widehat{\B(R)}}$-regular. $R\to \B(R)$ is pure if and only if $E\to \B(R)\otimes_R E$ is injective by \cite[Lemma 2.1 (e)]{HH95}. $R\to \B(R)$ is pure if and only if $R\to \B(R)$ is cyclically pure by \cite[Theorem 1.7]{cyclically pure}. Therefore, (5) is equivalent to (7).
\end{proof}
\begin{rem}
	For the equivalence of (5) and (7), see \cite[Proposition 6.14]{Ma-Schwede}.
\end{rem}
\section{Further questions and remarks}
In this section, we will consider whether $R$ is BCM-rational if $R$ has rational singularities.
The next question is a variant of \cite[Question 2.7]{seed}.
\begin{question}\label{question}
	Let $R$ be a local domain essentially of finite type over $\C$ and $B$ a big Cohen-Macaulay $R$-algebra. If $S$ is finitely generated $R$-algebra such that the following diagram commutes:
	\[
	\xymatrix{
		R \ar[r] \ar[d] & B \\
		S \ar[ru]
	}
	\]
	then does there exist a big Cohen-Macaulay $R_p$-algebra for almost all $p$ which fits into the following commutative diagram:
	\[
	\xymatrix{
		R_p \ar[r] \ar[d] & B_p \\
		S_p \ar[ru]
	}
	\]
	where $S_p$ is an $R$-approximation of $S$?
\end{question}
\begin{prop}[{cf. \cite[Conjecture 3.9]{Ma-Schwede}}]
	Let $R$ be a normal local domain essentially of finite type over $\C$ of dimension $d$. Suppose that $R$ has rational singularities. If Question \ref{question} has an affirmative answer, then $R$ is BCM-rational.
\end{prop}
\begin{proof}
	Let $B$ be a big Cohen-Macaulay $R^+$-algebra. Suppose that $\eta\in \Ker (H_m^d(R)\to H_\m^d(B))$. Then there exists a finitely generated $R$-subalgebra of $B$ such that the image of $\eta$ in $H_m^d(S)$ is zero. If Question \ref{question} has an affirmative answer, we can take $S_p$ and $B_p$ as in Question \ref{question}. Then we have a commutative diagram
	\[
	\xymatrix{
	H_\m^d(R) \ar[r] \ar[d] & \ulim_p H_{\m_p}^d(R_p) \ar[d] \\
	H_{\m}^d(S) \ar[r]\ar[d]& \ulim_p H_{\m_p}^d(S_p) \ar[d] \\
	H_m^d(B) & \ulim_p H_{\m_p}^d(B_p)
	}
	\]
	By the proof of Proposition \ref{rational and BCM rational}, $\ulim_p H_{\m_p}^d(R_p)\to \ulim_p H_{\m_p}^d(S_p)$ is injective. Therefore, the image of $\eta$ in $\ulim_p H_{\m_p}^d(R_p)$ is zero. Suppose that $\eta=\lbrack \frac{y}{x^t}\rbrack$, where $y\in R$, $t\in \N$ and $x$ is the product of a system of parameters $x_1,\dots,x_d$ of $R$. Since $R_p$ is Cohen-Macaulay for almost all $p$, $y_p\in(x_{1p}^t,\dots, x_{dp}^t)$ for almost all $p$. Hence, $y\in(x_1^t,\dots,x_d^t)$ and $\eta=0$ in $H_\m^d(R)$. Thus, $H_\m^d(R)\to H_\m^d(B)$ is injective.
\end{proof}
The next result follows from a similar argument.
\begin{prop}
	Let $R$ be a normal local domain essentially of finite type over $\C$ of dimension $d$. Fix an effective canonical divisor $K_R$ on $\Spec R$. Let $\Delta\ge 0$ be an effective $\Q$-Weil divisor on $\Spec R$ such that $K_R+\Delta$ is $\Q$-Cartier. Suppose that $C$ is a big Cohen-Macaulay $R^+$-algebra. If Question \ref{question} has an affirmative answer, then we have
	\[
		\J(R,\Delta)\subseteq \tau_{\widehat{C}}(\widehat{R},\widehat{\Delta}).
	\]
\end{prop}
\begin{defn}[cf. {\cite[Definition 6.9]{Ma-Schwede}}]
	Let $R$ be a normal local domain essentially of finite type over $\C$. Fix an effective canonical divisor $K_R$ on $\Spec R$. Let $\Delta\ge 0$ be a $\Q$-Weil divisor on $\Spec R$ such that $K_R+\Delta$ is $\Q$-Cartier. Suppose that $n(K_R+\Delta)=\Div(f)$ for $f\in R^\circ$, $n\in \N$. We define
	\begin{eqnarray*}
	0_{H_\m^d(R)}^{\mathscr{B},K_R+\Delta}:=\{\eta\in H_\m^d(R)| &\exists C \text{ big Cohen-Macaulay } R^+ \text{-algebra}&\\
	&{\text{such that }f^{\frac{1}{n}}\eta =0 \text{ in } H_\m^d(C)\}.}&
	\end{eqnarray*}
	We define {\it the BCM test ideal $\tau_{\mathscr{B}}(R,\Delta)$ of $(\widehat{R},\widehat{\Delta})$} by
	\[
	\tau_{\mathscr{B}}(\widehat{R},\widehat{\Delta}):=\Ann_{\omega_{\widehat{R}}}0_{H^d_\m(R)}^{\mathscr{B},K_R+\Delta}.
	\]
\end{defn}
\begin{cor}[cf. {\cite[Theorem 6.21]{Ma-Schwede}}]
	In the setting of the above proposition, if Question \ref{question} has an affirmative answer, then we have
	\[
	\tau_{\mathscr{B}}(\widehat{R},\widehat{\Delta})=\J(\widehat{R},\widehat{\Delta}).
	\]
\end{cor}


\end{document}